\documentclass[12pt,reqno,a4paper]{article}

\usepackage{lmodern}

\usepackage{graphicx}
\usepackage{amssymb}
\usepackage{amsmath}
\usepackage{amsthm}
\usepackage{url}
\usepackage{cite}

\usepackage{dsfont}
\usepackage{enumerate}
\usepackage{hyperref}
\usepackage{tikz}
\usepackage{float}
\usepackage{sidecap}
\usepackage{caption}

\usepackage{blindtext}
\usepackage[
    textcolor=red,
    linecolor=red,
    bordercolor=white,
    backgroundcolor=white,
]{todonotes}

\usepackage[title]{appendix}

\allowdisplaybreaks[3]

\newcommand{\dd}{\,\mathrm{d}}
\renewcommand{\Tilde}{\widetilde}
\renewcommand{\Hat}{\widehat}
\renewcommand{\Bar}{\overline}

\newcommand{\R}{\mathbb{R}}

\newcommand{\N}{\mathbb{N}}

\newcommand{\NN}{\mathbb{N}}

\newcommand{\eps}{\varepsilon}

\newcommand{\cH}{\mathcal{H}}

\newcommand{\U}{\mathcal{U}}
\newcommand{\V}{\mathcal{V}}

\newcommand{\spec}{\mathop{\mathrm{spec}}}
\newcommand{\ess}{\mathrm{ess}}

\newtheorem{theorem}{Theorem}[section]
\newtheorem{prop}[theorem]{Proposition}

\newtheorem{lemma}[theorem]{Lemma}
\theoremstyle{definition}
\newtheorem{rem}[theorem]{Remark}

\begin{document}

\title{\bf Asymptotics of Robin eigenvalues\\ for non-isotropic peaks}

\author{Marco Vogel\\[\smallskipamount]
	\small Carl von Ossietzky Universit\"at Oldenburg,\\
	\small Institut f\"ur Mathematik,\\
	\small 26111 Oldenburg, Germany\\[\smallskipamount]
	\small ORCID: 0009-0004-3628-0347\\
	\small E-Mail: \url{marco.vogel@uol.de}
}

\date{}

\maketitle

\begin{abstract}
\noindent Let $\Omega\subset \R^3$ be an open set such that
\begin{align*}
&\Omega \cap (-\delta,\delta)^3=\left\{(x_1,x_2,x_3)\in \R^2\times(0,\delta): \, \left(\frac{x_1}{x_3^p},\frac{x_2}{x_3^q}\right)\in(-1,1)^2\right\}\subset\R^{3}, 
\\
&\Omega \setminus [-\delta,\delta]^3 \text{ is a bounded Lipschitz domain},
\end{align*}
for some $\delta>0$ and $1<p<q<2$. If a set satisfies the first condition one says that it has a non-isotropic peak at $0$. Now consider the operator $Q_\Omega^\alpha$ acting as the Laplacian $u\mapsto-\Delta u$ on $\Omega$ with the Robin boundary condition $\partial_\nu u=\alpha u$ on $\partial\Omega$, where $\partial_\nu$ is the outward normal derivative. We are interested in the strong coupling asymptotics of $Q_\Omega^\alpha$. We prove that for large $\alpha$ the $j$th eigenvalue $E_j(Q_\Omega^\alpha)$ behaves as $E_j(Q_\Omega^\alpha)\approx \mathcal{A}_j\alpha^{\frac{2}{2-q}}$, where the constants $\mathcal{A}_j<0$ are eigenvalues of a one dimensional Schrödinger operator which depends on $p$ and $q$.
\end{abstract}


\smallskip

\section{Introduction}\label{intro}
Consider an open set $\Omega\subset\R^N$ and the Robin eigenvalue problem
\begin{align*}
-\Delta u&=\lambda u \quad \text{on }\Omega,
\\
\partial_\nu u&=\alpha u \quad \text{on }\partial\Omega,
\end{align*}
where  $\partial_\nu$ is the outward normal derivative and $\alpha>0$ is the so called Robin parameter, which is also referred to as a coupling constant. Numerous results concerning this problem have been published over the last decades. In \cite{bfk} the authors gave an overview of the current body of knowledge and also presented some open problems. However we are particularly interested in the following question: How do the eigenvalues behave as $\alpha$ tends to infinity? This is often referred to as the strong coupling asymptotics of the eigenvalues and was presumably first studied by Lacey, Ockendon and Sabina \cite{lacey}. For further discussion we need to define operators more rigorous. So let $\Omega\subset\R^N$ be an open set and $\alpha>0$ such that the quadratic form
\begin{equation*}
q_\Omega^\alpha(u,u)=\int_\Omega|\nabla u|^2\dd x-\alpha\int_{\partial\Omega}u^2\dd \sigma, \quad D(q_{\Omega}^{\alpha})=H^1(\Omega),
\end{equation*}
is semibounded from below and closed, where $\dd\sigma$ denotes the integration with respect to the $(N-1)$-dimensional Hausdorff measure, and denote by $Q_{\Omega}^{\alpha}$ the self-adjoint operator in $L^2(\Omega)$ associated with $q_\Omega^\alpha$. The asymptotic behavior of the eigenvalues is highly influenced by the regularity of $\Omega$. We list some results about the strong coupling regime for "nice" domains first and then move on to "bad" domains.

If $\Omega\subset\R^N$, with $N\ge2$, is a bounded Lipschitz domain, then it is well known that $0>E_j(Q_\Omega^\alpha)>-K\alpha^2$ for sufficiently large $\alpha>0$. Remark that the second inequality follows immediately from \cite[Theorem 1.5.1.10]{grisvard}. Levitin and Parnovski \cite[Theorem 3.2]{levitin} showed that the principal eigenvalue for bounded piecewise smooth domains satisfying the uniform interior cone condition behaves as $E_1(Q_\Omega^\alpha)\approx -C_\Omega\alpha^2$, with $C_\Omega\ge1$. If $\partial\Omega$ is $C^1$ then Daners and Kennedy \cite[Theorem 1.1.]{kennedy} were able to show that $C_\Omega=1$ for every eigenvalue, i.e. $E_j(Q_\Omega^\alpha)\approx-\alpha^2$. The results mentioned above are all one-term asymptotics, but there are also papers which have proven two-term asymptotics. Exner, Minakov and Parnovski \cite[Theorem 1.3]{exner} showed for planar domains, which have a closed $C^4$ Jordan curve as their boundary, that the eigenvalues behave as $E_j(Q_\Omega^\alpha)\approx-\alpha^2-\gamma^*\alpha$, where $\gamma^*$ is the maximal curvature of the mentioned Jordan curve. Another two-term asymptotic expansion was obtained for curvilinear polygons by Khalile, Ourmières-Bonafos and Pankrashkin, the corresponding paper is quite voluminous and technical, therefore we refer to \cite{konstantin} for precise statements.

For non Lipschitz domains many different scenarios are possible. If $\Omega$ has an outward pointing peak which is "to sharp" the Robin-Laplacian fails to be semibounded from below, see e.g. \cite[Lemma 1.2]{nazarov}. However the present paper is motivated by \cite{kov}, where Kova\v{r}\'{\i}k and Pankrashkin looked at isotropic peaks, i.e. there exists $\delta>0$ such that
\begin{align*}
&\Omega \cap (-\delta,\delta)^N=\left\{(x',x_N)\in \R^{N-1}\times(0,\delta): \, \frac{x'}{x_N^{q}}\in B_1(0)\right\}\subset\R^{N}, 
\\
&\Omega \setminus [-\delta,\delta]^N \text{ is a bounded Lipschitz domain},
\end{align*}
with $1<q<2$ and $B_1(0)$ being the unit ball centered at the origin in $\R^{N-1}$.
They proved that the rate of divergence of the eigenvalues to $-\infty$ is faster than in the pure Lipschitz case. In particular they showed, the eigenvalues behave as $E_j(Q_{\Omega}^\alpha)\approx\mathcal{E}_j\alpha^{\frac{2}{2-q}}$, with $\mathcal{E}_j<0$ being the $j$th eigenvalue of a one dimensional Schrödinger operator. One also observes that the sharper the peak the faster the divergence to $-\infty$. We change the premise of the aforementioned paper in the following way: Consider an open set $\Omega\subset \R^3$, which satisfies
\begin{align}
&\Omega \cap (-\delta,\delta)^3=\left\{(x_1,x_2,x_3)\in \R^2\times(0,\delta): \, \left(\frac{x_1}{x_3^p},\frac{x_2}{x_3^q}\right)\in(-1,1)^2\right\}\subset\R^{3}, \label{non iso peak}
\\
&\Omega \setminus [-\delta,\delta]^3 \text{ is a bounded Lipschitz domain},\label{lipschitz part}
\end{align}
for some $\delta>0$ and $1<p<q<2$. If a set satisfies condition \eqref{non iso peak} one says that it has a non-isotropic peak at $0$. Since $\Omega$ is not Lipschitz, the closedness and semiboundedness of the quadratic form $q_\Omega^\alpha$ are not obvious but will be justified in Appendix \ref{appb}.

Based on the above observation one might expect that the larger power $q$ determines the rate of divergence to $-\infty$, which turns out to be true as described in Theorem \ref{thm1}. For a precise statement we need to define a one dimensional Schrödinger operator. Consider the symmetric differential operator given by
\begin{equation*}
C_c^\infty(0,\infty)\ni f\mapsto -f''+\left(\frac{(p+q)^2-2(p+q)}{4s^2}-\frac{1}{s^q}\right)f
\end{equation*}
and denote by $A_{0,1}$ its Friedrichs extension in $L^2(0,\infty)$. Then the result reads as follows:
\begin{theorem}\label{thm1}
Let $j\in \N$ be fixed, then the $j$th eigenvalue of $Q_\Omega^\alpha$ satisfies
	\begin{equation*}
		  \label{asymp}
	E_j(Q_\Omega^\alpha)=\alpha^{\frac{2}{2-q}}E_j(A_{0,1})+O\left(\alpha^{\frac{2}{2-q}-(p-1)}+\alpha^{\frac{2}{2-q}-\frac{q-p}{2-q}}\right)\quad \text{as }\alpha\rightarrow\infty.
	\end{equation*}
\end{theorem}
\begin{rem}
Since $\Omega$ is bounded and has a continuous boundary, the embedding $H^1(\Omega)\hookrightarrow L^2(\Omega)$ is compact \cite[Ch. V Theorem 4.17]{edmunds}. Therefore $Q_\Omega^\alpha$ has compact resolvent, in particular the essential spectrum is empty and there exist infinitely many discrete eigenvalues.
\end{rem}

\begin{rem}
It would be desirable to substitute the "cross section" $(-1,1)^2$ with an arbitrary Lipschitz domain, but during our analysis this cross section is going to "collapse to an interval". So one has to analyze the behavior of the Robin eigenvalues on such a collapsing cross section, which may have its own interest and requires an independent study.
\end{rem}
The proof follows the same scheme as in \cite{kov, vogel} but with several additional technical ingredients, see section \ref{sec-iso peak} for further explanations.

\section{Definitions and auxiliary results}  \label{sec-prelim}

\subsection{Min-max principle}\label{ssmm}

Let $T$ be a lower semibounded, self-adjoint operator in an infinite-dimensional Hilbert space $\cH$. The essential spectrum of $T$ will be denoted by $\spec_\ess T$. Furthermore, denote $\Sigma:=\inf\spec_\ess T$ for $\spec_\ess T\ne \emptyset$ and $\Sigma:=+\infty$ otherwise. If $T$ has at least $j$ eigenvalues (counting multiplicities) in $(-\infty,\Sigma)$, then we denote by $E_j(T)$ its $j$th eigenvalue (when enumerated in the non-decreasing order and counted according to the multiplicities).  All operators we consider are real (i.e. map real-valued functions to real-valued functions), and we prefer to work with real Hilbert spaces in order to have shorter expressions. 

Let $t$ be the quadratic form for\, $T$, with domain $D(t)$, and let $D\subset D(t)$ be any dense subset (with respect to the scalar product induced by
$t$).
Consider the following ``variational eigenvalues''
\begin{equation*}
	\Lambda_j(T):=\inf_{\substack{V\subset D\\ \dim V=j}}\sup_{\substack{u\in V \\ u\neq 0}}\frac{t(u,u)}{\langle u,u\rangle_{\cH}}, 
\end{equation*}
which are independent of the choice of $D$. One easily sees that $j\mapsto \Lambda_j(T)$ is non-decreasing, and it is known \cite[Section XIII.1]{RS4} that only two cases are possible:
\begin{itemize}
	\item For all $j\in\mathbb{N}$ there holds $\Lambda_{j}(T)<\Sigma$. Then the spectrum of $T$ in $(-\infty,\Sigma)$ consists of infinitely
	many discrete eigenvalues $E_j(T)\equiv \Lambda_j(T)$ with $j\in\NN$.
	
	\item For some $N\in\NN\cup\{0\}$ there holds $\Lambda_{N+1}(T)\geq\Sigma$, while $\Lambda_j(T)<\Sigma$ for all $j\le N$.
	Then $T$ has exactly $N$ discrete eigenvalues in $(-\infty,\Sigma)$ and $E_j(T)=\Lambda_{j}(T)$ for $j\in\{1,\dots,N\}$, while $\Lambda_{j}(T)=\Sigma$ for all $j\geq N+1$.
\end{itemize}
In all cases there holds $\lim_{j\to\infty}\Lambda_j(T)=\Sigma$, and if for some $j\in\NN$ one has $\Lambda_{j}(T)<\Sigma$, then $E_j(T)=\Lambda_j(T)$.
In particular, if for some $j\in \NN$ one has the strict inequality $\Lambda_j(T)<\Lambda_{j+1}(T)$, then $E_j(T)=\Lambda_j(T)$.

\subsection{Isolating the peak}\label{isolateing peak}
Choose $\delta>0$ such that \eqref{non iso peak} holds. For such $\delta$ we define

\begin{equation}
\begin{aligned}
\label{omegadelta}\Omega_{\delta}&:=\left\{(x_1,x_2,x_3)\in\R^2\times (0,\delta):\left(\frac{x_1}{x_3^p},\frac{x_2}{x_3^q}\right)\in(-1,1)^2\right\},
\\
\Theta_\delta&:=\Omega\setminus\overline{\Omega}_{\delta},
\\
\partial_{0}\Omega_{\delta}&:=\{(x_1,x_2,x_3)\in\partial\Omega_{\delta}:x_3<\delta\},
\\
\widehat{H}_0^1(\Omega_{\delta})&:=\{u\in H^1(\Omega_\delta):u(\cdot,\delta)=0\}
\end{aligned}
\end{equation}
then $\overline{\Omega}=\overline{\Theta_{\delta}\cup\Omega_{\delta}}$ and $\Theta_{\delta}$ is a bounded Lipschitz domain.
\begin{figure}[H]
\begin{center}
\begin{tikzpicture}[x=0.75pt,y=0.75pt,yscale=-0.5,xscale=0.5]

\draw [color={rgb, 255:red, 0; green, 0; blue, 0 }  ,draw opacity=1 ]   (336.5,457) .. controls (328.5,395) and (278.37,287.76) .. (240.1,268.96) ;
\draw    (336.5,457) .. controls (340.5,380) and (347,332) .. (369.5,296.54) ;
\draw   (311.5,205.96) -- (440.9,233.54) -- (369.5,296.54) -- (240.1,268.96) -- cycle ;
\draw    (336.5,457) .. controls (336.5,361) and (324,247.5) .. (311.5,205.96) ;
\draw    (336.5,457) .. controls (345.5,370) and (384,296.5) .. (440.9,233.54) ;
\draw    (240.1,268.96) .. controls (292.5,87) and (440.5,193) .. (440.9,233.54) ;
\draw    (369.5,296.54) .. controls (388.5,250) and (326.5,95) .. (311.5,205.96) ;
\draw    (451.94,233.53) -- (452.5,459.78) ;
\draw    (440.9,233.54) -- (462.99,233.53) ;
\draw    (441.45,459.78) -- (463.55,459.78) ;

\draw (464,336) node [anchor=north west][inner sep=0.75pt]   [align=left] {$\displaystyle \delta $};
\draw (258,336) node [anchor=north west][inner sep=0.75pt]   [align=left] {$\displaystyle \Omega _{\delta }$};
\draw (220,191) node [anchor=north west][inner sep=0.75pt]   [align=left] {$\displaystyle \Theta _{\delta }$};

\end{tikzpicture}
	\caption{Truncation of  $\Omega$ in $\Omega_\delta$ and $\Theta_\delta$.}
\end{center}
\end{figure}
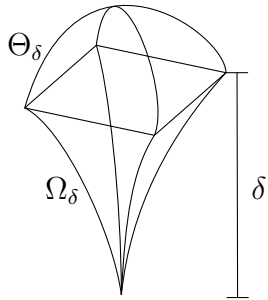
A standard application of the min-max principle shows that for any $j\in\N$ one has
\begin{equation}\label{sandwichRobin}
E_j(R_{\alpha}^{N,\Omega_{\delta}}\oplus K_{\alpha}^{N,\Theta_\delta})\le E_j(Q^{\alpha}_{\Omega})\le E_j(R_{\alpha}^{D,\Omega_{\delta}})
\end{equation}
where $R_{\alpha}^{N/D,\Omega_{\delta}}$ are the self-adjoint operators in $L^2(\Omega_{\delta})$ defined respectively by the quadratic forms
\begin{align*}
r_\alpha^{N,\Omega_{\delta}}(u,u)&=\int_{\Omega_{\delta}}|\nabla u|^2 \dd x-\alpha\int_{\partial_{0}\Omega_{\delta}}u^2\dd \sigma,\quad &&D(r_\alpha^{N,\Omega_{\delta}})=H^1(\Omega_{\delta}),
\\
r_\alpha^{D,\Omega_{\delta}}(u,u)&=r_\alpha^{N,\Omega_{\delta}}(u,u),\quad  &&D(r_\alpha^{D,\Omega_{\delta}})=\widehat{H}_0^1(\Omega_{\delta}),
\end{align*}
and $K_{\alpha}^{N,\Theta_\delta}$ is the self-adjoint operator in $L^2(\Theta_\delta)$ defined by the quadratic form
\begin{equation*}
k_{\alpha}^{N,\Theta_\delta}(u,u)=\int_{\Theta_{\delta}}|\nabla u|^2 \dd x-\alpha\int_{\partial\Omega\cap\partial\Theta_{\delta}}u^2\dd \sigma,\quad D(k_{\alpha}^{N,\Theta_\delta})=H^1(\Theta_{\delta}).
\end{equation*}
Standard dilation arguments show the unitary equivalence
$R_{\alpha}^{N/D,\Omega_{\delta}}\simeq \alpha^{2}R^{N/D,\alpha\Omega_{\delta}}_{1}$ and we remark that

\[
\alpha\Omega_\delta=\left\{(x_1,x_2,x_3)\in\R^2\times(0,\alpha\delta):\left(\alpha^{p-1}\frac{x_1}{x_3^p},\alpha^{q-1}\frac{x_2}{x_3^q}\right)\in(-1,1)^2\right\}.
\]

\subsection{Robin Laplacians on intervals} 
\label{sec-1d}

Given $r\in\R$ and $L>0$ denote by $B_{L,r}$ the self-adjoint operator in $L^2(-L,L)$ generated by the closed quadratic form 
\begin{equation*} \label{d-form} 
	b_{L, r}(f,f) =\int_{-L}^L |f'(t)|^2\, \dd t   -  r (|f(L)|^2+|f(-L)|^2),  \qquad  D(b_{L,r})= H^1(-L,L).
\end{equation*} 
The operator $B_{L,r}$  is the Laplacian $f\mapsto -f''$ on $H^2(-L,L)$
with the Robin boundary condition $f'(\pm L) =\pm r f(\pm L)$ and we will summarize some important spectral properties of $B_{L,r}$ as follows.

\begin{lemma} \label{lem-1}
	
	The following assertions hold true:
	
	\begin{enumerate}
		\item [(a)] $E_j(B_{L,r})=\frac{1}{L^2}E_j(B_{1,rL})$ for all $L>0$ and $r\in\R$.
		\item [(b)] $E_1(B_{L,r})=-\frac{r}{L}+r^2\phi(rL)$ with $\phi\in C^{\infty}(\R_+)\cap L^{\infty}(\R_+)$ for all $L,r>0$.
		\item [(c)] $E_2(B_{L,r})=\frac{\pi^2}{4L^2}-\frac{2r}{L}+r^2\phi(rL)$ with $\phi\in C^{\infty}(\R_+)\cap L^{\infty}(\R_+)$ for all $L,r>0$.
		\item [(d)] Let $j\in\N$ be fixed, then there exists a family of normalized eigenfunctions $(\Phi_{j,L,r})_{r\in\R}$ corresponding to $(E_j(B_{L,r}))_{r\in\R}$ such that $\R\ni r\mapsto\Phi_{j,L,r}\in L^2(-L,L)$ is $C^\infty$, which we call a smooth family of eigenfunctions. Furthermore for all $r_0>0$ and all $C>0$ there exists $C_0>0$ such that 
		\[ \int_{-L}^{L}|\partial_r\Phi_{j,L,r}(t)|^2\dd t\le C_0L^2,\]
		 for all $L\le C$ and all $r\in(0,r_0)$.

	\end{enumerate}	
\end{lemma} 

\begin{proof}
	To prove the claim in $(a)$ one can use the unitary operator $\mathcal{U}:L^2(-L,L)\ni f\mapsto \sqrt{L}f(L\cdot)\in L^2(-1,1)$ to show that $\mathcal{U}B_{L,r}\,\mathcal{U}^{-1}=\frac{1}{L^2}B_{1,rL}$.

	Regarding $(b)$: In \cite[Proposition 2.2]{kp18} it was shown that there exists $\phi\in C^{\infty}(\R_+)\cap L^{\infty}(\R_+)$ such that for all $r>0$ there holds $E_1(B_{1,r})=-r+r^2\phi(r)$. To conclude the proof of $(b)$ we use $E_1(B_{L,r})=\frac{1}{L^2}E_1(B_{1,rL})$ from $(a)$.
	
	For $(c)$ we define $\phi(r):=r^{-2}(E_2(B_{1,r})-\frac{\pi^2}{4}+2r)$. First of all the function $\R\ni r\mapsto E_2(B_{1,r})$ is analytic, since $B_{1,r}$ is a type (B) analytic family and has compact resolvent and $E_2(B_{1,r})$ has multiplicity one for all $r$, which implies that $\phi$ is smooth, see e.g. \cite[Remark 4.22 in Ch. 7]{kato}. To show the boundedness of $\phi$ we prove that it admits finite limits, i.e. as $r\rightarrow 0$ and $r\rightarrow \infty$. For $r\rightarrow 0$ we would like to have a Taylor expansion of $r\mapsto E_2(B_{1,r})$ in $0$. According to \cite[Equation (4.12)]{bfk} we can use an eigenfunction corresponding to the second Neumann eigenvalue $E_2(B_{1,0})=\frac{\pi^2}{4}$ to get an expression for $\partial_r E_2(B_{1,r})|_{r=0}$. Namely
	\[
	\partial_r E_2(B_{1,r})|_{r=0}=-\frac{|\Phi_{2,1,0}(1)|^2+|\Phi_{2,1,0}(-1)|^2}{\|\Phi_{2,1,0}\|^2_{L^2(-1,1)}},
	\]
where $\Phi_{2,1,0}(t)=\sin\left(\frac{\pi t}{2}\right)$ is a corresponding eigenfunction for $E_2(B_{1,0})$ and therefore $\partial_r E_2(B_{1,r})|_{r=0}=-2$. It follows that $E_2(B_{1,r})=\frac{\pi^2}{4}-2r+O(r^2)$ as $r\rightarrow 0$ and therefore $\phi(r)=O(1)$ as $r\rightarrow0$. For the case $r\rightarrow\infty$ we use the result from $(a)$ to get $E_2(B_{1,r})=r^2E_2(B_{r,1})$. Now we can use \cite[Proposition A.3.]{hp}, which states that $E_2(B_{r,1})=-1+O(r^2e^{-2r})$ as $r\rightarrow\infty$. Plugging these results back into $\phi$ yields $\phi(r)=O(1)$ as $r\rightarrow\infty$. Once again we use the result from $(a)$ to conclude the proof of $(c)$.
	
	Lastly we fix some $j\in\N$. The existence of the family $(\Phi_{j,L,r})_{r\in\R}$ mentioned in $(d)$ also follows from $B_{L,r}$ being a type (B) analytic family, see again \cite[Remark 4.22 in Ch. 7]{kato}. To prove the second claim in $(d)$, choose a family $(\Phi_{j,1,r})_{r\in\R}$ such that $r\mapsto \Phi_{j,1,r}$ is smooth. Then for all $\widetilde{r}_0>0$ there exists $C_0>0$ such that
	\begin{equation}	\label{bddEF}	
	 \int_{-1}^{1}|\partial_r\Phi_{j,1,r}(t)|^2\dd t\le C_0\quad \quad \text{ for all }r\in(0,\widetilde{r}_0).
	 \end{equation}
Now let $L\le C$, $r_0>0$ and $\widetilde{r}_0 := Cr_0$. Use the unitary operator $\U$ mentioned in the proof of $(a)$ to get $\Phi_{j,L,r}=\U^{-1}\Phi_{j,1,Lr}=L^{-\frac{1}{2}}\Phi_{j,1,Lr}(\frac{\cdot}{L})$. By equation \eqref{bddEF} there exists $C_0$ such that
	\begin{align*}	
	\int_{-L}^{L}|\partial_r\Phi_{j,L,r}(t)|^2\dd t&=L\int_{-L}^{L}\left(\partial_{\rho}\Phi_{j,1,\rho}\left(\frac{t}{L}\right)\bigg|_{\rho=Lr}\right)^2\dd t
	\\	
	&=L^{2}\int_{-1}^{1}\left(\partial_{\rho}\Phi_{j,1,\rho}(t)\big|_{\rho=Lr}\right)^2\dd t\le C_0L^2.
	\end{align*}
	holds.
\end{proof}
For further reference we need the following Lemma.
\begin{lemma}\label{lem2}
Let $[a,b]\subset \R$ be an interval and $\rho,\widetilde\rho\in C^1([a,b])$ with $\rho,\widetilde\rho\ge0$. Then for all $j,k\in\N$ and all $C>0$ there exists $C_0>0$ such that 
\begin{equation*}
\int_{-L}^L\int_{-\widetilde L}^{\widetilde L}|\partial_s (\Phi_{j,L,\rho(s)}(t_1)\Phi_{k,\widetilde L,\widetilde\rho(s)}(t_2))|^2\dd t_2\dd t_1\le C_0(L^2+\widetilde L^2),
\end{equation*}
for all $s\in[a,b]$ and all $L,\widetilde L\le C$. Here $(\Phi_{j,L,r})_{r\in\R}$ and $(\Phi_{k,\widetilde L,r})_{r\in\R}$ are smooth families of normalized eigenfunctions as defined in Lemma \ref{lem-1} (d).
\end{lemma}
\begin{proof}
Let $C>0$ and $L,\widetilde L\le C$. By Young's inequality we have
\begin{equation*}
\begin{aligned}
\left[\partial_s (\Phi_{j,L,\rho(s)}(t_1)\Phi_{k,\widetilde L,\widetilde\rho(s)}(t_2))\right]^2=&\big[\rho'(s)\partial_r \Phi_{j,L,r}(t_1)\big|_{r=\rho(s)}\Phi_{k,\widetilde L,\widetilde\rho(s)}(t_2)
\\
&+\widetilde\rho\ '(s)\Phi_{j,L,\rho(s)}(t_1)\partial_r\Phi_{k,\widetilde L,r}(t_2)\big|_{r=\widetilde\rho(s)}\big]^2
\\
\le& 2\left[\rho'(s)\partial_r \Phi_{j,L,r}(t_1)\big|_{r=\rho(s)}\Phi_{k,\widetilde L,\widetilde\rho(s)}(t_2)\right]^2
\\
&+2\left[\widetilde\rho\ '(s)\Phi_{j,L,\rho(s)}(t_1)\partial_r\Phi_{k,\widetilde L,r}(t_2)\big|_{r=\widetilde\rho(s)}\right]^2,
\end{aligned}
\end{equation*}
for all $t_1\in[-L,L]$, $t_2\in[-\widetilde L, \widetilde L]$ and $s\in[a,b]$.
Furthermore there exists $C_1>0$ and $C_0>0$ such that for all $s\in[a,b]$
\begin{align*}
\int_{-L}^L\int_{-\widetilde L}^{\widetilde L}|\partial_s (&\Phi_{j,L,\rho(s)}(t_1)\Phi_{k,\widetilde L,\widetilde\rho(s)}(t_2))|^2\dd t_2\dd t_1
\\
\le\int_{-L}^L\int_{-\widetilde L}^{\widetilde L}2\Big[&\rho'(s)\partial_r \Phi_{j,L,r}(t_1)\big|_{r=\rho(s)}\Phi_{k,\widetilde L,\widetilde\rho(s)}(t_2)\Big]^2
\\
&+2\Big[\widetilde\rho\ '(s)\Phi_{j,L,\rho(s)}(t_1)\partial_r\Phi_{k,\widetilde L,r}(t_2)\big|_{r=\widetilde\rho(s)}\Big]^2\dd t_2 \dd t_1
\\
=\int_{-L}^L 2\Big[\rho'(s)&\partial_r \Phi_{j,L,r}(t_1)\big|_{r=\rho(s)}\Big]^2\dd t_1+\int_{-\widetilde L}^{\widetilde L}2\Big[\widetilde\rho\ '(s)\partial_r\Phi_{k,\widetilde L,r}(t_2)\big|_{r=\widetilde\rho(s)}\Big]^2\dd t_2
\\
\le C_1 \int_{-L}^L \Big(&\partial_r \Phi_{j,L,r}(t_1)\big|_{r=\rho(s)}\Big)^2\dd t_1+C_1\int_{-\widetilde L}^{\widetilde L}\Big( \partial_r\Phi_{k,\widetilde L,r}(t_2)\big|_{r=\widetilde\rho(s)}\Big)^2\dd t_2
\\
\le C_0(L^2+&\widetilde L^2).
\end{align*}
For the equality we used that $\Phi_{j,L,r}$ and $\Phi_{k,\widetilde L,r}$ are normalized eigenfunctions, for the second inequality we used that $\rho'$ and $\widetilde\rho\ '$ are bounded  and for the last inequality we used that $\rho$ and $\widetilde \rho$ are bounded combined with Lemma \ref{lem-1} $(d)$.
\end{proof}

\subsection{One-dimensional model operators}\label{compop}
Let $c_1,c_2\ge 0$ such that $c_1+c_2>0$ and consider the symmetric differential operator in $L^2(0,\infty)$ given by
\begin{equation*}
	C_c^\infty(0,\infty) \ni f \mapsto -f'' + \bigg(\frac{(p+q)^2-2(p+q)}{4s^2}-\frac{c_1}{s^p}-\frac{c_2}{s^q}\bigg) f.
\end{equation*}
Due to $(p+q)^2-2(p+q)=((p+q)-1)^2-1$ and the Hardy inequality
\begin{equation*}
\int_0^\infty |f'|^2\dd s\ge\int_0^\infty \frac{f^2}{4s^2}\dd s\quad \text{for $f\in C_c^\infty(0,\infty)$}
\end{equation*}
the operator is semibounded from below.
Denote by $A_{c_1,c_2}$ its Friedrichs extension and by $a_{c_1,c_2}$ its corresponding quadratic form. One can show that $A_{c_1,c_2}$ has infinitely many negative eigenvalues and for $c>0$ there holds the unitary equivalence
\begin{equation}\label{uemodel}
A_{c_1,c_2}\simeq c^2A_{c_1c^{p-2},c_2c^{q-2}}.
\end{equation}
In what follows we will need to work with truncated versions of $A_{c_1,c_2}$. Namely for $b>0$ we denote by $M_{c_1,c_2,b}$ and $\widetilde{M}_{c_1,c_2,b}$ the Friedrichs extensions in $L^2(0,b)$ and $L^2(b,\infty)$ of the operators $C_{c}^\infty (0,b)\ni f\mapsto A_{c_1,c_2} f$ and $C_{c}^\infty (b,\infty)\ni f\mapsto A_{c_1,c_2}f$ respectively. Remark that by construction the form domain of $M_{c_1,c_2,b}$ is contained in $H_0^1(0,b)$, which implies that $M_{c_1,c_2,b}$ has compact resolvent. We need to relate the eigenvalues of $M_{c_1,c_2,b}$ to those of $A_{c_1,c_2}$. As the quadratic form of $A_{c_1,c_2}$ extends that of $M_{c_1,c_2,b}$, one has, due to the min-max principle,
\begin{equation}\label{model lowbd}
E_j(M_{c_1,c_2,b})\geq E_j(A_{c_1,c_2})\quad \text{ for any }\, b>0, j\in\NN.
\end{equation}
Let us now obtain an asymptotic upper bound for $E_j(M_{c_1,c_2,b})$.
\begin{lemma}\label{modelop1}
Let $b>0$ and $j\in\NN$. There exist $K>0$ and $\alpha_0>0$ such that
\[
E_j(M_{\alpha^{p-1},\alpha^{q-1},b})\le E_j(A_{\alpha^{p-1},\alpha^{q-1}})+K\quad \text{for all }\alpha>\alpha_0.
\]
\end{lemma}
\begin{proof}
	The proof is quite standard and uses a so-called IMS partition of unity \cite[Sec.~3.1]{bsim}.
	Let $\chi_1$ and $\chi_2$ be two smooth functions on $\R$ with $0\le\chi_1,\chi_2\le 1$, such that
	$\chi_1^2+\chi_2^2=1$, $\chi_1(s)=0$ for $s>\frac{3}{4}\, b$, $\chi_2(s)=0$ for $s<\frac{1}{2}\, b$.
	We set $K:=\|\chi_1'\|_\infty^2+\|\chi_2'\|_\infty^2$.
	An easy computation shows that
	\begin{align*}
		\int_0^\infty |f'|^2 \dd s & =\int_0^\infty \big|(\chi_1 f)'\big|^2 \dd s +\int_0^\infty \big|(\chi_2 f)'\big|^2 \dd s
		-\int_0^\infty \Big(|\chi_1'|^2+|\chi_2'|^2\Big) f^2\, \dd s\\
		& \ge \int_0^\infty \big|(\chi_1 f)'\big|^2 \dd s +\int_0^\infty \big|(\chi_2 f)'\big|^2 \dd s
		- K \|f\|^2_{L^2(0,\infty)},
	\end{align*}
	for any $f\in C^\infty_c(0,\infty)$, which implies
	\begin{align*}
		\big\langle f&, A_{\alpha^{p-1},\alpha^{q-1}} f\big\rangle_{L^2(0,\infty)}+K \|f\|^2_{L^2(0,\infty)}
		\\
		&\ge \big\langle \chi_1 f, A_{\alpha^{p-1},\alpha^{q-1}}(\chi_1 f)\big\rangle_{L^2(0,\infty)}+\big\langle 
		\chi_2 f, A_{\alpha^{p-1},\alpha^{q-1}}(\chi_2 f)\big\rangle_{L^2(0,\infty)}\\
		\\
		&= \big\langle \chi_1 f, A_{\alpha^{p-1},\alpha^{q-1}}(\chi_1 f)\big\rangle_{L^2(0,b)}+\big\langle 
		\chi_2 f, A_{\alpha^{p-1},\alpha^{q-1}}(\chi_2 f)\big\rangle_{L^2(\frac{b}{4},\infty)}.
	\end{align*}
	Using the identity $\|f\|^2_{L^2(0,\infty)}=\|\chi_1 f\|^2_{L^2(0,b)}+\|\chi_2 f\|^2_{L^2(\frac{b}{4},\infty)}$
	and the obvious inclusions $\chi_1 f\in C^\infty_c(0,b)$, $\chi_2 f\in C^\infty_c(\frac{b}{4},\infty)$,
	we apply the min-max principle as follows:
	\begin{equation}
		\label{eq-ineq00}
		\begin{aligned}
			&E_j(A_{\alpha^{p-1},\alpha^{q-1}})\!+\!K =\! \inf_{\substack{S\subset C^\infty_c(0,\infty)\\ \dim S=j}} \sup_{\substack{f\in S\\ f\ne 0}} \dfrac{\langle f, A_{\alpha^{p-1},\alpha^{q-1}} f\rangle_{L^2(0,\infty)}+K\|f\|^2_{L^2(0,\infty)}}{\|f\|^2_{L^2(0,\infty)}}\\
			& \ge\! \inf_{\substack{S\subset C^\infty_c(0,\infty)\\\dim S=j}} \sup_{\substack{f\in S \\ f\ne 0}} \dfrac{\big\langle \chi_1 f, A_{\alpha^{p-1},\alpha^{q-1}}\!(\chi_1 f)\big\rangle_{L^2(0,b)}\!+\big\langle \chi_2 f, A_{\alpha^{p-1},\alpha^{q-1}}(\chi_2 f)\big\rangle_{L^2(\frac{b}{4},\infty)}}{\|\chi_1 f\|^2_{L^2(0,b)}+\|\chi_2 f\|^2_{L^2(\frac{b}{4},\infty)}} 
			\\
			& \ge \!\inf_{\substack{S\subset C^\infty_c(0,b)\times C^\infty_c(\frac{b}{4},\infty)\\ \dim S=j}}\!  \sup_{\substack{(f_1,f_2) \in S\\ (f_1,f_2)\ne 0}} \dfrac{\big\langle f_1 , A_{\alpha^{p-1},\alpha^{q-1}}\, f_1\big\rangle_{L^2(0,b)}
				\!\!+\!\big\langle f_2 , A_{\alpha^{p-1},\alpha^{q-1}}\, f_2\big\rangle_{L^2(\frac{b}{4},\infty)}}{\|f_1\|^2_{L^2(0,b)}+\|f_2\|^2_{L^2(\frac{b}{4},\infty)}}\\
			& =\Lambda_j\big(M_{\alpha^{p-1},\alpha^{q-1},b}  \oplus \Tilde M_{\alpha^{p-1},\alpha^{q-1}, \frac{b}{4}}\big)\!\ge\! \min\big\{\!\Lambda_j(M_{\alpha^{p-1},\alpha^{q-1},b}),\Lambda_1(\Tilde M_{\alpha^{p-1},\alpha^{q-1}, \frac{b}{4}})\big\}.
		\end{aligned}
	\end{equation}
For any $j\in\NN$ we have $\Lambda_j(M_{\alpha^{p-1},\alpha^{q-1},b})=E_j(M_{\alpha^{p-1},\alpha^{q-1},b})$. At the same time, for any function $f\in C_c^\infty(\frac{b}{4},\infty)$ one has
	\begin{align*}
		\langle f, \Tilde M_{\alpha^{p-1},\alpha^{q-1}, \frac{b}{4}} f\rangle_{L^2(\frac{b}{4},\infty)}&=\int_{\frac{b}{4}}^\infty|f'|^2+
		\left[\frac{(p+q)^2-2(p+q)}{4s^2}-\frac{\alpha^{p-1}}{s^p}-\frac{\alpha^{q-1}}{s^q}\right] f^2\dd s\\
		&\ge -\left(\frac{4^p}{b^p}\alpha^{p-1}+\frac{4^q}{b^q}\alpha^{q-1}\right)\|f\|^2_{L^2(\frac{b}{4},\infty)},
	\end{align*}
	which gives the lower bound $\Lambda_1(\Tilde M_{\alpha^{p-1},\alpha^{q-1}, \frac{b}{4}})\ge  -\frac{4^p}{b^p}\alpha^{p-1}-\frac{4^q}{b^q}\alpha^{q-1}$. Due to 
\[	
	A_{\alpha^{p-1},\alpha^{q-1}}\le A_{0,\alpha^{q-1}}\quad \text{and}\quad A_{0,\alpha^{q-1}}\simeq \alpha^{2\frac{q-1}{2-q}} A_{0,1}
\]	
	 we conclude that if $j\in\NN$ is fixed, then one can find some $\alpha_0>0$
	such that for all $\alpha>\alpha_0$ there holds $E_j(A_{\alpha^{p-1},\alpha^{q-1}})+K<\Lambda_1(\Tilde M_{\alpha^{p-1},\alpha^{q-1}, \frac{b}{4}})$.
	Then \eqref{eq-ineq00} implies $E_j(A_{\alpha^{p-1},\alpha^{q-1}})+K\ge E_j(M_{\alpha^{p-1},\alpha^{q-1},b})$.
\end{proof}

\subsection{Scheme of the proof}\label{sec-iso peak}

We are interested in large Robin parameters, so from now on we only consider $\alpha\ge 1$. Let us also remark that the same dilation arguments mentioned in Subsection \ref{isolateing peak} can be used to show $Q_\Omega^\alpha\simeq \alpha^2 Q_{\alpha\Omega}^1$. The peak will determine the asymptotic behavior of the eigenvalues, so we start by restricting the Robin-Laplacian $Q_{\alpha\Omega}^1$ to the peak of $\alpha\Omega$. More precisely, pick some $0<a<\min(\delta,1)$ (this value will remain fixed through the whole text), and denote
\begin{align*}
	V_{\alpha}&:=\alpha\Omega\cap(-a,a)^3\subset\R^{3},\\
	\partial_0 V_{\alpha}&:=\{(x_1,x_2,x_3)\in\partial V_\alpha:x_3<a\}\subset \partial V_{\alpha},    \\
	\Hat H^1_0(V_{\alpha})&:=\{u\in H^1(V_\alpha):u(\cdot,a)=0\},\\
	C^\infty_{(0,a)}(\Bar V_\alpha)&:=\big\{ u\in C^\infty(\Bar V_\alpha):\exists\, [b,c]\subset (0, a)\;\text{s.t.}\;\text{$u(x_1,x_2,x_3)=0$ for $x_3\notin [b,c]$} \big\}.
\end{align*}
Let $T_\alpha$ be the self-adjoint operator in $L^2(V_{\alpha})$ associated with the quadratic form 
\begin{equation}
	\label{teps1}
	t_\alpha (u,u):=\int_{V_{\alpha}} |\nabla u|^2\dd x- \int_{\partial_0 V_{\alpha}} u^2\, \dd\sigma, \quad D(t_\alpha)=\Hat H^1_0(V_{\alpha}),
\end{equation}
then $T_\alpha$ can be informally interpreted as the Laplacian in $V_\alpha$ with the Robin boundary condition $\partial_\nu u=u$ on $\partial_0 V_\alpha$ and the Dirichlet boundary condition on the remaining boundary $\partial V_\alpha\setminus \partial_0 V_\alpha$ (which corresponds to $x_3=a$). In view of Lemma~\ref{prop4dens} $C^\infty_{(0,a)}(\Bar V_\alpha)$ is dense in $\Hat H^1_0(V_{\alpha})$, so by the min-max principle the variational eigenvalues of $T_\alpha$ are given by
\begin{equation} \label{ej-vp0}
	E_j(T_\alpha) = \inf_{\substack{S\subset D_0(t_\alpha) \\ {\rm dim\, } S =j} } \, \sup_{\substack{u\in S\\ u\neq 0}}\,  \frac{t_\alpha(u,u)}{\quad \|u\|_{L^2(V_{\alpha})}^2}\,,
	\quad
	D_0(t_\alpha):=C^\infty_{(0,a)}(\Bar V_\alpha),
	\quad j\in\NN.
\end{equation}
Using  a suitable change of coordinates and the spectral analysis of $B_{L,r}$, the study of eigenvalues of $T_\alpha$ for large $\alpha$ is reduced to the one-dimensional operators $A_{c_1,c_2}$.
In Section \ref{ssthm1} we use \eqref{sandwichRobin} as a starting point to prove that the eigenvalues of $T_\alpha$ and $Q_{\alpha\Omega}^1$ are asymptotically close to each other.

\section{Spectral analysis near the peak}\label{sec-teps}

\subsection{Change of variables} \label{sec-subs0}
We need suitable coordinates on $\partial_0V_\alpha$, therefore we transform coordinates similar to~\cite{kov}. 
Consider the diffeomorphism
\begin{equation*}\label{diffeo}
X:\R^2\times(0,\infty)\to \R^2\times(0,\infty), \quad
X(t_1,t_2,s)=(s^p t_1,s^q t_2,s)
\end{equation*}
then one checks that
\begin{gather*}
	V_{\alpha} = X (\Pi_{\alpha}), \quad \Pi_{\alpha} := \omega_{\alpha}\times (0,a), \quad
\omega_{\alpha}:=(-\alpha^{1-p},\alpha^{1-p})\times (-\alpha^{1-q},\alpha^{1-q}),
\\
\partial_0 V_\alpha=X \big((\partial\omega_\alpha)\times(0,a) \big) \cup\{0\}.
\end{gather*}
Remark that $\{0\}$ has zero two-dimensional Hausdorff measure and can be neglected in the integration over $\partial_0V_\alpha$.
This induces the unitary transform (change of variables)
\begin{equation*}\label{unitaryU}
\U: L^2(V_{\alpha})\to L^2(\Pi_\alpha, s^{p+q}\dd s\,\dd t),
\quad
\U\, u := u\circ X,
\end{equation*} 
where we used the notation $t=(t_1,t_2)$. Consider the quadratic form $r_\alpha$
given by 
\[
r_\alpha(u,u) := t_\alpha (\U^{-1} u,\U^{-1} u), \quad D(r_\alpha):=\U D(t_\alpha),
\]
in $L^2(\Pi_\alpha, s^{p+q}\dd s\,\dd t)$. Due to the unitarity of $\U$ and Lemma~\ref{prop4dens}, the subspace
\begin{align*}
D_0(r_\alpha):=&\,\U\, D_0(t_\alpha)\\
=&  \big\{u\in C^\infty(\overline{\Pi}_{\alpha}):\, \exists\, [b,c]\subset (0,a) \text{ such that } u(t,s)=0 \text{ for } s\notin[b,c]\big\},
\end{align*}
is dense in $D(r_\alpha)$, and by \eqref{ej-vp0} one has
\begin{equation} \label{ej-vp2a}
E_j(T_\alpha) = \inf_{\substack{S\subset D_0(r_\alpha) \\ {\rm dim\, } S =j} } \, \sup_{\substack{u\in S\\ u\neq 0}}\,  \frac{r_\alpha(u,u)}{\quad \|u\|_{L^2(\Pi_{\alpha}, s^{p+q}\dd s\,\dd t)}^2}.
\end{equation} 

Now we would like to obtain a more convenient expression for $r_\alpha(u,u)$.
\begin{lemma}\label{lem5}
For any $v\in D_0(t_\alpha)$ and $u:=\U v\in D_0(r_\alpha)$,
\begin{multline*}
\int_0^a \int_{\omega_{\alpha}}
\Big[\big(1-\alpha^{1-p}-\alpha^{1-q} \big)\, |\partial_s u|^2 + \dfrac{1-p^2\alpha^{1-p}}{s^{2p}}\, |\partial_{t_{1}} u|^2+ \dfrac{1-q^2\alpha^{1-q}}{s^{2q}}\, |\partial_{t_{2}} u|^2\Big]s^{p+q}\dd t \dd s\\
\leq \int_{V_\alpha} |\nabla v|^2\dd x
\\
\leq \int_0^a \int_{\omega_{\alpha}}
\Big[\big(1+\alpha^{1-p}+\alpha^{1-q} \big)\, |\partial_s u|^2 + \dfrac{1+p^2(\alpha^{1-p}+2\alpha^{2-2p})}{s^{2p}}\, |\partial_{t_{1}} u|^2
\\
+ \dfrac{1+q^2(\alpha^{1-q}+2\alpha^{2-2q})}{s^{2q}}\, |\partial_{t_{2}} u|^2\Big]s^{p+q}\dd t \dd s.
\end{multline*}
\end{lemma}

\begin{proof}
A standard computation shows that for any $u\in D_0(r_\alpha)$ and $v:=\U^{-1} u$ there holds
\begin{equation}
	 \label{eq-form1a}
\int_{V_\alpha} |\nabla v|^2\dd x=\int_0^a \int_{\omega_\alpha}  \langle \nabla u, G\,  \nabla u \rangle_{\R^{3}}\,  s^{p+q}  \dd t \dd s
\end{equation}
where $G$ is the matrix given by
\[
G= (D X^T \, DX)^{-1}
\]
with $DX$ being the Jacobi matrix of $X$.
One checks directly that
\begin{equation*} 
G(t_1,t_2,s)= \left(\, 
 \begin{matrix}
s^{-2p}+p^2 s^{-2}t_1^2&  qps^{-2}t_1 t_2& -ps^{-1}t_1\,    \\[\medskipamount]
qps^{-2}t_1 t_2& s^{-2q}+q^2 s^{-2}t_2^2& -qs^{-1}t_2\\
-ps^{-1}t_1& -qs^{-1}t_2& 1
\end{matrix}\, \right)\,.
\end{equation*}
We would like to estimate the term $\langle \nabla u,G\,  \nabla u \rangle_{\R^{3}}$ from above and from below using simpler expressions. One obtains
\begin{equation}
\begin{aligned}
	\langle \nabla u, G\,  &\nabla u \rangle_{\R^{3}}
	=(s^{-2p}+p^2 t_1^2 s^{-2})|\partial_{t_1}\!u|^2+2pq s^{-2} t_1 t_2\, \partial_{t_1}\!u\, \partial_{t_2}\!u \\		&+\!(s^{-2q}\!+q^2 t_2^2\, s^{-2})|\partial_{t_2}\!u|^2\!-2p s^{-1} t_1 \partial_{t_1}\!u\, \partial_{s}u\!-2q s^{-1} t_2\, \partial_{t_2}\!u\, \partial_{s}u+|\partial_s u|^2.
\label{scalar prod}
\end{aligned}
\end{equation}
Using the standard inequalities $2|xy|\le x^2+y^2,\, |t_1|\le\alpha^{1-p},\, |t_2|\le\alpha^{1-q}$ and $0<s<a<1$ we estimate
\begin{align*}
	|2pq s^{-2} t_1 t_2 \,\partial_{t_1}\!u\, \partial_{t_2}\!u|&\leq \alpha^{2-2p}p^2 s^{-2p}|\partial_{t_1}\!u|^2+\alpha^{2-2q}q^2 s^{-2q} |\partial_{t_2}\!u|^2,
	\\
	|2p s^{-1} t_1 \partial_{t_1}\!u\, \partial_{s}u|&\leq \alpha^{1-p} (p^2 s^{-2p} |\partial_{t_1}\!u|^2 + |\partial_{s}u|^2),
	\\
	|2q s^{-1} t_2\, \partial_{t_2}\!u\, \partial_{s}u|&\leq \alpha^{1-q} (q^2 s^{-2q} |\partial_{t_2}\!u|^2 + |\partial_{s}u|^2).
\end{align*}
The substitution into \eqref{scalar prod} gives a two-sided estimate for $\langle \nabla u, G\,  \nabla u\rangle_{\R^{3}}$, and the substitution into \eqref{eq-form1a} gives the claim.
\end{proof}

Also remark that for the boundary we have

\begin{equation*}
\begin{gathered}
\int_{\partial_0 V_\alpha}\!|v|^2\! \dd\sigma \!=\!\int_0^a\!\bigg[ s^q \sqrt{p^2 \alpha^{2-2p} s^{2p-2}\!+\!1}\int_{-\alpha^{1-q}}^{\alpha^{1-q}}\!|u(\alpha^{1-p},t_2,s)|^2\!+\!|u(-\alpha^{1-p},t_2,s)|^2\dd t_2 
\\
+ s^p \sqrt{q^2 \alpha^{2-2q} s^{2q-2}+1}\int_{-\alpha^{1-p}}^{\alpha^{1-p}}|u(t_1,\alpha^{1-q},s)|^2+|u(t_1,-\alpha^{1-q},s)|^2\dd t_1\bigg]\dd s,
\end{gathered}
\end{equation*}
\begin{equation}\label{bdry}
\begin{gathered}
s^q\le s^q \sqrt{p^2\alpha^{2-2p}s^{2p-2}+1}\le s^q\sqrt{p^2\alpha^{2-2p}+1},
\\
s^p\le s^p \sqrt{q^2\alpha^{2-2q}s^{2q-2}+1}\le s^p \sqrt{q^2\alpha^{2-2q}+1},
\end{gathered}
\end{equation}
again for $u\in D_0(r_\alpha)$ and $v:=\U^{-1} u$.
By applying Lemma \ref{lem5} and \eqref{bdry} to both summands of $t_\alpha$ in \eqref{teps1}
and by adjusting various constants
we obtain the following two-sided estimate
written in a form adapted for the subsequent analysis: There exists $c>0$ (this value will remain fixed through the whole text) and $\alpha_0>1$ such that for any $u\in D_0(r_\alpha)$ and all $\alpha>\alpha_0$ there holds
\begin{gather*}
r_\alpha^-(u,u)\le r_\alpha(u,u)\le r_\alpha^+(u,u),\\
\begin{aligned}
r_\alpha^-(u,u):=(1- &c\alpha^{1-p})\int_0^a\bigg[ \int_{\omega_\alpha}\left( \dfrac{|\partial_{t_1}\! u|^2}{s^{2p}} + \dfrac{|\partial_{t_2}\! u|^2}{s^{2q}}+ |\partial_s u|^2\right) s^{p+q}\dd t_1 \dd t_2
\\ 
 &- \frac{s^q}{1- c\alpha^{1-p}}\int_{-\alpha^{1-q}}^{\alpha^{1-q}}|u(\alpha^{1-p},t_2,s)|^2+|u(-\alpha^{1-p},t_2,s)|^2\dd t_2
\\ 
 &-  \frac{s^p}{1- c\alpha^{1-p}}\int_{-\alpha^{1-p}}^{\alpha^{1-p}}|u(t_1,\alpha^{1-q},s)|^2+|u(t_1,-\alpha^{1-q},s)|^2\dd t_1\bigg] \dd s,
\end{aligned}
\\
 \begin{aligned}
r_\alpha^+(u,u):=(1+ &c\alpha^{1-p})\int_0^a \bigg[\int_{\omega_\alpha}\left( \dfrac{|\partial_{t_1}\! u|^2}{s^{2p}} + \dfrac{|\partial_{t_2}\! u|^2}{s^{2q}}+|\partial_s u|^2\right) s^{p+q}\dd t_1 \dd t_2
\\ 
 &-\frac{s^q}{1+ c\alpha^{1-p}} \int_{-\alpha^{1-q}}^{\alpha^{1-q}}|u(\alpha^{1-p},t_2,s)|^2+|u(-\alpha^{1-p},t_2,s)|^2\dd t_2
\\ 
 &-\frac{s^p}{1+ c\alpha^{1-p}}  \int_{-\alpha^{1-p}}^{\alpha^{1-p}}|u(t_1,\alpha^{1-q},s)|^2+|u(t_1,-\alpha^{1-q},s)|^2\dd t_1\bigg] \dd s.
\end{aligned}
\end{gather*}
Remark that $c>0$ is independent of the choice of $0<a<\min(\delta,1)$. With the help of the unitary transform
\[
\V: L^2(\Pi_\alpha)\to L^2(\Pi_\alpha, s^{p+q} \dd s\, \dd t), \quad
(\V\, u)(t,s) = s^{-\frac{p+q}{2}}\,  u(t,s),
\]
we define the following quadratic forms, which we will work with most of the time.
\begin{align*}
p_\alpha^-(u,u):=&r_\alpha^-(\V^{-1}u,\V^{-1}u),
\\
p_\alpha^+(u,u):=&r_\alpha^+(\V^{-1}u,\V^{-1}u),
\end{align*}
with $u\in D_0(p_\alpha^-):=\V^{-1}D_0(r_\alpha)=D_0(r_\alpha)=:D_0(p_\alpha^+)$. Using elementary calculations the following Proposition follows.

\begin{prop}\label{prop6}
For all $u\in D_0(p_\alpha^\pm)$ there holds
\begin{align*}
p_\alpha^-(u,u)=&(1- c\alpha^{1-p})\int_0^a\bigg[ \int_{\omega_\alpha} \dfrac{|\partial_{t_1}\! u|^2}{s^{2p}}
 + \dfrac{|\partial_{t_2}\! u|^2}{s^{2q}}
\\
&+ |\partial_s u|^2+\dfrac{(p+q)^2-2(p+q)}{4s^2}|u|^2\dd t_1 \dd t_2
\\
&- \dfrac{1}{(1- c\alpha^{1-p})s^p}\int_{-\alpha^{1-q}}^{\alpha^{1-q}}|u(\alpha^{1-p},t_2,s)|^2+|u(-\alpha^{1-p},t_2,s)|^2\dd t_2
\\
&- \dfrac{1}{(1- c\alpha^{1-p})s^q} \int_{-\alpha^{1-p}}^{\alpha^{1-p}}|u(t_1,\alpha^{1-q},s)|^2+|u(t_1,-\alpha^{1-q},s)|^2\dd t_1 \bigg]\dd s,
\\
p_\alpha^+(u,u)=&(1+ c\alpha^{1-p})\int_0^a\bigg[ \int_{\omega_\alpha} \dfrac{|\partial_{t_1}\! u|^2}{s^{2p}}+ \dfrac{|\partial_{t_2}\! u|^2}{s^{2q}}
\\
&+ |\partial_s u|^2+\dfrac{(p+q)^2-2(p+q)}{4s^2} |u|^2\dd t_1 \dd t_2
\\
&- \dfrac{1}{(1+ c\alpha^{1-p})s^p}\int_{-\alpha^{1-q}}^{\alpha^{1-q}}|u(\alpha^{1-p},t_2,s)|^2+|u(-\alpha^{1-p},t_2,s)|^2\dd t_2
\\
&- \dfrac{1}{(1+ c\alpha^{1-p})s^q} \int_{-\alpha^{1-p}}^{\alpha^{1-p}}|u(t_1,\alpha^{1-q},s)|^2+|u(t_1,-\alpha^{1-q},s)|^2\dd t_1 \bigg]\dd s.
\end{align*}
In particular, by \eqref{ej-vp2a} it follows that
for each $j\in\NN$ there holds
\[
\inf_{\substack{{S\subset D_0(p_\alpha^-)} \\ {\rm dim\,}  S =j} } \, \sup_{\substack{u\in S\\ u\neq 0}}\,  \frac{p_\alpha^-(u,u)}{\quad \|u\|_{L^2(\Pi_\alpha)}^2}\le	E_j(T_\alpha) \le \inf_{\substack{S\subset D_0(p_\alpha^+) \\ {\rm dim\,}  S =j }} \, \sup_{\substack{u\in S\\ u\neq 0}}\,  \frac{p_\alpha^+(u,u)}{\quad \|u\|_{L^2(\Pi_\alpha)}^2}.
\]
\end{prop}

\subsection{Upper bound for the eigenvalues of $T_\alpha$} \label{sec-upperb}
We are going to compare the eigenvalues of $T_\alpha$ with those of the one-dimensional operators $A_{c_1,c_2}$.

\begin{lemma}\label{lem6}
Let $j\in\N$ then there exist $C>0$ and $\alpha_0\ge1$ such that for any $\alpha>\alpha_0$ there holds
\begin{equation*}
E_j(T_\alpha)\le (1+ c\alpha^{1-p})E_j\left(M_{\frac{\alpha^{p-1}}{1+ c\alpha^{1-p}},\frac{\alpha^{q-1}}{1+c\alpha^{1-p}},a}\right)+C.
\end{equation*}
\end{lemma}

\begin{proof}
By Proposition \ref{prop6} we have
\begin{equation*}
E_j(T_\alpha)\le \inf_{\substack{S\subset D_0(p^+_\alpha) \\ {\rm dim\,}  S =j }} \, \sup_{\substack{u\in S\\ u\neq 0}}\,  \frac{p_\alpha^+(u,u)}{\quad \|u\|_{L^2(\Pi_\alpha)}^2}.
\end{equation*}
We want to have a more convenient expression for $p_\alpha^+$ to compare it to an one dimensional Robin-Laplacian. There holds
\begin{align*}
p_\alpha^+&(u,u)=(1+c\alpha^{1-p})\int_0^a \bigg[\int_{-\alpha^{1-q}}^{\alpha^{1-q}}\dfrac{1}{s^{2p}}
\\
&\cdot\left\{\int_{-\alpha^{1-p}}^{\alpha^{1-p}}|\partial_{t_1} u|^2 \dd t_1 - \dfrac{s^p}{1+c\alpha^{1-p}}\left(|u(\alpha^{1-p}, t_2,s)|^2+|u(-\alpha^{1-p},t_2,s)|^2\right)\right\}\dd t_2
\\
&+\int_{-\alpha^{1-p}}^{\alpha^{1-p}}\dfrac{1}{s^{2q}}
\\
&\cdot\left\{\int_{-\alpha^{1-q}}^{\alpha^{1-q}}|\partial_{t_2} u|^2 \dd t_2 -\dfrac{s^q}{1+c\alpha^{1-p}}\left(|u(t_1,\alpha^{1-q},s)|^2+|u(t_1,-\alpha^{1-q},s)|^2\right)\right\}\dd t_1
\\
&+\int_{\omega_{\alpha}}|\partial_s u|^2+\frac{(p+q)^2-2(p+q)}{4s^2}|u|^2\dd t\bigg]\dd s.
\end{align*}

Note that the functionals in the curly brackets are the quadratic forms $b_{\alpha^{1-k},\rho_k(s)}$ as defined in Subsection \ref{sec-1d} with 
\[
\rho_k:s \mapsto\dfrac{s^k}{1+c\alpha^{1-p}}\in C^1([0,a])\cap C^\infty(0,a),\quad k\in\{p,q\}.
\]
Let $k\in\{p,q\}$ and $(\Phi_{1,k,\rho_k(s)})_{s\in(0,a)}$ be a smooth family of normalized eigenfunctions corresponding to $(E_1(B_{\alpha^{1-k},\rho_{k}(s)}))_{s\in(0,a)}$ (see Lemma \ref{lem-1} $(d)$).
Moreover, if $S \subset C_c^\infty(0,a)$ is a $j$-dimensional subspace, then
\begin{align*} 
\Tilde S = \big\{ u\in H^1(\Pi_\alpha):u(t,s)= f(s)\,\Phi_{1,p,\rho_p(s)}(t_1)\Phi_{1,q,\rho_q(s)}(t_2)\text{ with }f\in S\big\} 
\end{align*} 
is a $j$-dimensional subspace of $D_0(p_\alpha^+)$. For any $u\in \Tilde S$  one has $\|u\|^2_{L^2(\Pi_\alpha)} = \|f\|^2_{L^2(0,a)}$ and 
\begin{align*}
\int_{-\alpha^{1-q}}^{\alpha^{1-q}}&\int_{-\alpha^{1-p}}^{\alpha^{1-p}}|\partial_{t_1} u|^2 \dd t_1 - \rho_p(s)\left(|u(\alpha^{1-p}, t_2,s)|^2+|u(-\alpha^{1-p},t_2,s)|^2\right)\dd t_2 
\\
&= E_1(B_{\alpha^{1-p},\rho_p(s)})|f(s)|^2,
\\
\int_{-\alpha^{1-p}}^{\alpha^{1-p}}&\int_{-\alpha^{1-q}}^{\alpha^{1-q}}|\partial_{t_2} u|^2 \dd t_2 - \rho_q(s)\left(|u(t_1,\alpha^{1-q},s)|^2+|u(t_1,-\alpha^{1-q},s)|^2\right)\dd t_1 
\\
&= E_1(B_{\alpha^{1-q},\rho_q(s)})|f(s)|^2.
\end{align*}
For $u\in\widetilde S$ this simplifies the expression of the above quadratic form as follows
\begin{align*}
 p_\alpha^+(u,u)=(1+c\alpha^{1-p})\int_0^a&\bigg( \left[\dfrac{E_1(B_{\alpha^{1-p},\rho_{p}(s)})}{s^{2p}}+\dfrac{E_1(B_{\alpha^{1-q},\rho_{q}(s)})}{s^{2q}}\right]\, |f(s)|^2 
\\
&+\int_{\omega_{\alpha}}|\partial_s u|^2+\frac{(p+q)^2-2(p+q)}{4s^2}|u|^2\dd t\bigg)\dd s.
\end{align*}
For the sake of brevity we use the notation 
\begin{align*}
\Phi(t_1,t_2,s) &:= \Phi_{1,p,\rho_p(s)}(t_1)\Phi_{1,q,\rho_q(s)}(t_2).
\end{align*}
Due to the normalization of $\Phi$ we have
\[
\int_{\omega_\alpha} 2\Phi\, \partial_s\Phi \dd t=\partial_s \int_{\omega_\alpha} |\Phi|^2\dd t=\partial_s 1=0.
\]
The preceding orthogonality relation and Lemma \ref{lem2} show
\begin{align*}
\int_{\omega_\alpha}  |\partial_s u|^2\dd t&=\int_{\omega_\alpha} |f'|^2|\Phi|^2 +2f'f\Phi\partial_s\Phi+|f|^2|\partial_s \Phi|^2\dd t
\\
&=|f'|^2+\int_{\omega_\alpha}|f|^2|\partial_s\Phi|^2\dd t\le |f'|^2+ C (\alpha^{2-2p}+\alpha^{2-2q})|f|^2
\end{align*}
for appropriate $C>0$.
Using Lemma \ref{lem-1} and taking $\alpha$ sufficiently large, we arrive at
\begin{align*} 
p_\alpha^+(u,u)\!\leq\!(1\!+\!c\alpha^{1-p})\!\int_{0}^a& \!\bigg(\!\left[\dfrac{-\rho_p(s)\alpha^{p-1}+C_1|\rho_p(s)|^2}{s^{2p}}+\dfrac{-\rho_q(s)\alpha^{q-1}+C_1|\rho_q(s)|^2}{s^{2q}}\right]|f|^2
\\
&+\!|f'|^2\!+\!\frac{(p+q)^2-2(p+q)}{4s^2}|f|^2+C (\alpha^{2-2p}+\alpha^{2-2q})|f|^2\bigg)\dd s
\\
\le (1+ c\alpha^{1-p})\int_{0}^a&\bigg(|f'|^2+\bigg[ \frac{(p+q)^2-2(p+q)}{4s^2}-\dfrac{\alpha^{p-1}}{(1+ c\alpha^{1-p})s^{p}}
\\
&-\dfrac{\alpha^{q-1}}{(1+ c\alpha^{1-p})s^{q}}\bigg]|f|^2\bigg)\dd s+C\| u\|^2_{L^2(\Pi_\alpha)}
\end{align*}
for appropriate $C$ and $C_1$.  One recognizes the quadratic form for $M_{\frac{\alpha^{p-1}}{1+ c\alpha^{1-p}},\frac{\alpha^{q-1}}{1+ c\alpha^{1-p}},a}$. Also remark that the constants we have chosen, in this proof, to estimate terms are independent of the subspace $S$. As a result we have
\begin{align*}
E_j(T_\alpha)&\le \inf_{\substack{S\subset D_0(p^+_\alpha) \\ {\rm dim\, }S =j }} \, \sup_{\substack{u\in S\\ u\neq 0}}\,  \frac{p_\alpha^+(u,u)}{\quad \|u\|_{L^2(\Pi_\alpha)}^2}\leq \inf_{\substack{S\subset C_c^\infty(0,a)\\ {\rm dim\, }S =j }} \, \sup_{\substack{u\in \tilde S\\ u\neq 0}}\,  \frac{p_\alpha^+(u,u)}{\quad \|u\|_{L^2(\Pi_\alpha)}^2}
\\
&\le(1+ c\alpha^{1-p})\inf_{\substack{S\subset C_c^\infty(0,a) \\ {\rm dim\, }S =j }} \, \sup_{\substack{f\in S\\ f\neq 0}}\dfrac{\langle f, M_{\frac{\alpha^{p-1}}{1+ c\alpha^{1-p}},\frac{\alpha^{q-1}}{1+ c\alpha^{1-p}},a}f\rangle}{\| f\|_{L^2(0,a)}^2}+C.
\end{align*}
This implies
\begin{align*}
E_j(T_\alpha)&\le (1+ c\alpha^{1-p})E_j\left(M_{\frac{\alpha^{p-1}}{1+ c\alpha^{1-p}},\frac{\alpha^{q-1}}{1+c\alpha^{1-p}},a}\right)+C
\end{align*}
and concludes the proof.

\end{proof}

\begin{prop}\label{prop7}
For any $j\in\N$ there exists $\alpha_0,C>0$ such that 
\begin{equation*}
E_j(T_\alpha)\le \alpha^{2\frac{q-1}{2-q}}E_j(A_{0,1})+C\alpha^{2\frac{q-1}{2-q}-(p-1)}
\end{equation*}
for all $\alpha>\alpha_0$.
\end{prop}
\begin{proof}
By Lemma \ref{lem6}, we know
\begin{equation*}
E_j(T_\alpha)\le (1+ c\alpha^{1-p})E_j\left(M_{\frac{\alpha^{p-1}}{1+ c\alpha^{1-p}},\frac{\alpha^{q-1}}{1+ c\alpha^{1-p}},a}\right)+C
\end{equation*}
and, by Lemma \ref{modelop1}, we get
\begin{equation*}
E_j\left(M_{\frac{\alpha^{p-1}}{1+ c\alpha^{1-p}},\frac{\alpha^{q-1}}{1+ c\alpha^{1-p}},a}\right)\le E_j\left(A_{\frac{\alpha^{p-1}}{1+ c\alpha^{1-p}},\frac{\alpha^{q-1}}{1+ c\alpha^{1-p}}}\right)+K.
\end{equation*}
Applying the min-max principle yields
\begin{equation*}
E_j\left(A_{\frac{\alpha^{p-1}}{1+ c\alpha^{1-p}},\frac{\alpha^{q-1}}{1+ c\alpha^{1-p}}}\right)\le E_j\left(A_{0,\frac{\alpha^{q-1}}{1+ c\alpha^{1-p}}}\right).
\end{equation*}
Making use of the unitary equivalence \eqref{uemodel}, we have
\begin{align*}
E_j\left( A_{0,\frac{\alpha^{q-1}}{1+ c\alpha^{1-p}}}\right)=\alpha^{2\frac{q-1}{2-q}} E_j(A_{0,1})+O\left(\alpha^{2\frac{q-1}{2-q}-(p-1)}\right).
\end{align*}
By adjusting $C>0$ it follows that
\begin{equation*}
E_j(T_\alpha)\le \alpha^{2\frac{q-1}{2-q}}E_j(A_{0,1})+C\alpha^{2\frac{q-1}{2-q}-(p-1)}
\end{equation*}

\end{proof}

\subsection{Lower bound for the eigenvalues of $T_\alpha$} \label{lowbdd}

The lower bound for the eigenvalues of $T_\alpha$ is also obtained using a comparison with the operators $A_{c_1,c_2}$
but requires more work.

\begin{lemma}\label{lem9}
Let $j\in\NN$, then there exist $C>0$ and $\alpha_0>0$ such that 
\[
E_j(T_\alpha)\ge (1-C\alpha^{1-p})E_j\left(M_{\frac{\alpha^{p-1}}{1-C\alpha^{1-p}},\frac{\alpha^{q-1}}{1-C\alpha^{1-p}},a}\right)- C
\]
for all $\alpha>\alpha_0$.
\end{lemma}

\begin{proof}
By Proposition \ref{prop6} we have
\begin{equation*}
E_j(T_\alpha)\ge \inf_{\substack{S\subset D_0(p^-_\alpha) \\ {\rm dim\,  }S =j }} \, \sup_{\substack{u\in S\\ u\neq 0}}\,  \frac{p_\alpha^-(u,u)}{\quad \|u\|_{L^2(\Pi_\alpha)}^2}.
\end{equation*}
We want to have a more convenient expression for $p_\alpha^-$ to compare it to an one dimensional Robin-Laplacian. There holds
\begin{align*}
p_\alpha^-&(u,u)=(1-c\alpha^{1-p})\int_0^a\bigg[ \int_{-\alpha^{1-q}}^{\alpha^{1-q}}\dfrac{1}{s^{2p}}
\\
&\cdot\bigg\{\int_{-\alpha^{1-p}}^{\alpha^{1-p}}|\partial_{t_1} u|^2 \dd t_1 - \dfrac{s^p}{1-c\alpha^{1-p}}\left(|u(\alpha^{1-p}, t_2,s)|^2+|u(-\alpha^{1-p},t_2,s)|^2\right)\bigg\}\dd t_2
\\
&+\int_{-\alpha^{1-p}}^{\alpha^{1-p}}\dfrac{1}{s^{2q}}
\\
&\cdot\bigg\{\int_{-\alpha^{1-q}}^{\alpha^{1-q}}|\partial_{t_2} u|^2 \dd t_2 -\dfrac{s^q}{1-c\alpha^{1-p}}\left(|u(t_1,\alpha^{1-q},s)|^2+|u(t_1,-\alpha^{1-q},s)|^2\right)\bigg\}\dd t_1
\\
&+\int_{\omega_{\alpha}}|\partial_s u|^2+\frac{(p+q)^2-2(p+q)}{4s^2}|u|^2\dd t\bigg]\dd s
\end{align*}
Note that the functionals in the curly brackets are the quadratic forms $b_{\alpha^{1-k},\rho_k(s)}$ as defined in Subsection \ref{sec-1d} with 
\[
\rho_k:s\mapsto\dfrac{s^k}{1-c\alpha^{1-p}}\in C^1([0,a])\cap C^\infty(0,a),\quad k\in\{p,q\}.
\]
 Let $k\in\{p,q\}$ and $(\Phi_{1,k,\rho_k(s)})_{s\in(0,a)}$ be a smooth family of normalized eigenfunctions corresponding to $(E_1(B_{\alpha^{1-k},\rho_{k}(s)}))_{s\in(0,a)}$ (see Lemma \ref{lem-1} $(d)$).
We decompose each $u\in D_0(p_\alpha^-)$ as
\begin{align*}
u =&\ v + w  \ \text{ with } 
\\ 
v(t,s)=&f(s)  \Phi_{1,p,\rho_p(s)}(t_1)\Phi_{1,q,\rho_q(s)}(t_2),
\\
f(s) :=& \int_{\omega_\alpha} u(t,s)\, \Phi_{1,p,\rho_p(s)}(t_1)\Phi_{1,q,\rho_q(s)}(t_2)\, \dd t.
\end{align*}
By construction we have $f\in C^\infty_c(0,a)$ and, furthermore,
\begin{equation} \label{orth}
\begin{aligned}
\int_{\omega_\alpha} w(t,s)\, v(t,s)\, \dd t &= 0 \quad \text{ for any $s\in (0,a)$,}\\
\|f\|^2_{L^2(0,a)}&=\|v\|^2_{L^2(\Pi_\alpha)},\\
\|u\|^2_{L^2(\Pi_\alpha)}&=\|f\|^2_{L^2(0,a)} + \|w\|^2_{L^2(\Pi_\alpha)}.  
\end{aligned}
\end{equation}
The spectral theorem and Lemma \ref{lem-1} applied to $\dfrac{1}{s^{2p}}B_{\alpha^{1-p},\rho_{p}(s)}\otimes 1 + 1 \otimes\dfrac{1}{s^{2q}}B_{\alpha^{1-q},\rho_{q}(s)}$ implies that there exists $\alpha_0$ such that for all $\alpha>\alpha_0$ and any $u\in D_0(p_\alpha^-)$ we have
\begin{align*}
   \label{eq-utw}
\int_{-\alpha^{1-q}}^{\alpha^{1-q}}\dfrac{1}{s^{2p}}&
\\
&\cdot\Big\{ \int_{-\alpha^{1-p}}^{\alpha^{1-p}}|\partial_{t_1} u|^2 \dd t_1 - \dfrac{s^p}{1-c\alpha^{1-p}}\left(|u(\alpha^{1-p}, t_2,s)|^2+|u(-\alpha^{1-p},t_2,s)|^2\right)\Big\}\dd t_2
\\
+\int_{-\alpha^{1-p}}^{\alpha^{1-p}}&\dfrac{1}{s^{2q}}
\\
&\cdot\Big\{  \int_{-\alpha^{1-q}}^{\alpha^{1-q}}|\partial_{t_2} u|^2 \dd t_2 -\dfrac{s^q}{1-c\alpha^{1-p}}\left(|u(t_1,\alpha^{1-q},s)|^2+|u(t_1,-\alpha^{1-q},s)|^2\right)\Big\}\dd t_1
\\
\ge\, \bigg[  &\dfrac{E_1(B_{\alpha^{1-p},\rho_{p}(s)})}{s^{2p}}+\dfrac{E_1(B_{\alpha^{1-q},\rho_{q}(s)})}{s^{2q}}\bigg]\, |f(s)|^2 
\\
+&\left[\dfrac{E_2(B_{\alpha^{1-p},\rho_{p}(s)})}{s^{2p}}+\dfrac{E_1(B_{\alpha^{1-q},\rho_{q}(s)})}{s^{2q}}\right]\, \int_{\omega_\alpha} |w(t,s)|^2\, \dd t \, .
\end{align*}
By Lemma \ref{lem-1} we can find constants $C_i>0$, such that
\begin{equation}\label{minmax}
\begin{aligned}
\bigg[&\dfrac{E_1(B_{\alpha^{1-p},\rho_{p}(s)})}{s^{2p}}+\dfrac{E_1(B_{\alpha^{1-q},\rho_{q}(s)})}{s^{2q}}\bigg]\, |f(s)|^2 
\\
&+\left[\dfrac{E_2(B_{\alpha^{1-p},\rho_{p}(s)})}{s^{2p}}+\dfrac{E_1(B_{\alpha^{1-q},\rho_{q}(s)})}{s^{2q}}\right]\, \int_{\omega_\alpha} |w(t,s)|^2\, \dd t \,
\\
\ge &\left[-\dfrac{\rho_p(s)\alpha^{p-1}+C_1|\rho_p(s)|^2}{s^{2p}}-\dfrac{\rho_q(s)\alpha^{q-1}+C_1|\rho_q(s)|^2}{s^{2q}}\right]|f(s)|^2 
\\
&+\!\!\left[\!\dfrac{\frac{\pi^2}{4}\alpha^{2p-2}\!-\!2\rho_{p}(s)\alpha^{p-1}\!\!-\!C_1|\rho_{p}(s)|^2}{s^{2p}}-\dfrac{\rho_q(s)\alpha^{q-1}\!\!+\!C_1|\rho_q(s)|^2}{s^{2q}}\right]\!\!\int_{\omega_\alpha}\!\!\! |w(t,s)|^2\! \dd t
\\
\ge& \left[-\dfrac{\alpha^{p-1}}{(1-c\alpha^{1-p})s^{p}}-\dfrac{\alpha^{q-1}}{(1-c\alpha^{1-p})s^{q}}\right]|f(s)|^2 
\\
&+\left[\frac{\pi^2\alpha^{2p-2}}{4s^{2p}}-\dfrac{2\alpha^{p-1}}{(1-c\alpha^{1-p})s^{p}}-\dfrac{\alpha^{q-1}}{(1-c\alpha^{1-p})s^{q}}\right]\int_{\omega_\alpha} |w(t,s)|^2\dd t
\\
&-C_2 \left(|f(s)|^2+\int_{\omega_\alpha} |w(t,s)|^2\dd t\right)
\end{aligned}
\end{equation}
for all $u\in D_0(p_\alpha^-)$. 

Let us now study the integral
\begin{equation}\label{sabh}
\int_0^a\int_{\omega_{\alpha}}|\partial_s u|^2+\frac{(p+q)^2-2(p+q)}{4s^2}|u|^2\dd t\dd s.
\end{equation}
For the sake of brevity we will denote 
\begin{align*}
\Phi(t_1,t_2,s) &:= \Phi_{1,p,\rho_p(s)}(t_1)\Phi_{1,q,\rho_q(s)}(t_2),\quad v_s:=\partial_s v, \quad w_s:=\partial_s w\, .
\end{align*}
Using the orthogonality relation from \eqref{orth} we obtain
\begin{equation}\label{us}
\begin{aligned}
\int_0^a\!\!\int_{\omega_\alpha}&  |\partial_s u|^2+\frac{(p+q)^2-2(p+q)}{4 s^2} \, u^2  \dd t \dd s   
\\
= \int_0^a\!\! &\int_{\omega_\alpha} v_s^2+\frac{(p+q)^2-2(p+q)}{4 s^2} \, v^2  \dd t \dd s+ 2 \int_0^a\!\! \int_{\omega_\alpha} v_s\, w_s\, \dd t\, \dd s
\\
&+\int_0^a\!\! \int_{\omega_\alpha}  w_s^2+\frac{(p+q)^2-2(p+q)}{4 s^2} \, w^2  \dd t \dd s .
\end{aligned}
\end{equation}
Due to the normalization of $\Phi$ we have
\[
\int_{\omega_\alpha} 2\Phi\, \partial_s\Phi \dd t=\partial_s \int_{\omega_\alpha} |\Phi|^2\dd t=\partial_s 1=0.
\]
The preceding orthogonality relation shows the following
\begin{align*}
\int_{\omega_\alpha}  v_s^2\dd t&=\int_{\omega_\alpha} |f'|^2|\Phi|^2+2f'f\Phi\,\partial_s\Phi+|f|^2|\partial_s \Phi|^2\dd t
\\
&=|f'|^2+\int_{\omega_\alpha}|f|^2|\partial_s\Phi|^2\dd t\ge|f'|^2
\end{align*}
and, consequently,
\begin{align*}
\int_0^a \int_{\omega_\alpha}  v_s^2+\frac{(p+q)^2-2(p+q)}{4 s^2} \, v^2  \dd t \dd s \geq  \int_0^a  \, |f'|^2 + \frac{(p+q)^2-2(p+q)}{4s^2}\, f^2  \dd s. \label{vs-lowerb}
\end{align*}
In order to estimate the two last terms in \eqref{us} we note that  
\begin{equation*} \label{crossed}
 2\int_0^a \int_{\omega_\alpha} v_s\, w_s\, \dd t \dd s  =  2\int_0^a \int_{\omega_\alpha} f'\, \Phi\, w_s + f\, (\partial_s \Phi)\, w_s \dd t\dd s
\end{equation*}
and that $\langle w,\Phi\rangle_{L^2(\omega_\alpha)}=0$, which implies
\[
\int_{\omega_\alpha}\Phi\, w_s\dd t = -\int_{\omega_\alpha}(\partial_s \Phi)\, w\dd t.
\]
Hence,
\begin{equation*}
2\int_0^a \int_{\omega_\alpha} v_s\, w_s\, \dd t \dd s  =  2\int_0^a \int_{\omega_\alpha}  f\, (\partial_s \Phi)\, w_s - f'\, (\partial_s\Phi)\, w \dd t\dd s.
\end{equation*}
Due to Lemma \ref{lem2} and Young's inequality, there exists a $C>0$ such that for large $\alpha$ there holds
\begin{align*}
\int_0^a \int_{\omega_\alpha}  |f\, (\partial_s \Phi)\, w_s|\dd t\dd s&\leq C(\alpha^{2-2p}+\alpha^{2-2q})\int_0^a  |f|^2 \dd s + \frac{1}{2} \| w_s\|^2_{L^2(\Pi_\alpha)},
\\
\int_0^a \int_{\omega_\alpha}  |f'\, (\partial_s\Phi)\, w|\dd t\dd s&\leq C(\alpha^{2-2p}+\alpha^{2-2q})\int_0^a  |f'|^2 \dd s + \frac{1}{2} \| w\|^2_{L^2(\Pi_\alpha)}.
\end{align*}
Combining these inequalities gives
\begin{align*}
2&\int_0^a\! \int_{\omega_\alpha} |v_s\, w_s| \dd t \dd s \leq 2C(\alpha^{2-2p}+\alpha^{2-2q}) \int_0^a \!|f|^2\!+\!|f'|^2\dd s+ \| w_s\|^2_{L^2(\Pi_\alpha)} \!+ \| w\|^2_{L^2(\Pi_\alpha)}.
\end{align*}
Finally we get an estimation for \eqref{sabh}
\begin{equation}\label{born-op}
\begin{aligned}
\int_0^a\int_{\omega_{\alpha}}&|\partial_s u|^2+\frac{(p+q)^2-2(p+q)}{4s^2}|u|^2\dd t\dd s
\\
\ge \int_0^a&\bigg[ |f'|^2+\dfrac{(p+q)^2-2(p+q)}{4s^2}\left(|f|^2+\int_{\omega_\alpha}w^2\dd t\right)+\int_{\omega_\alpha}w_s^2\dd t
\\
&-\widetilde C\left(\alpha^{2-2p}+\alpha^{2-2q})(|f|^2+|f'|^2\right)-\int_{\omega_\alpha}w_s^2\dd t-\int_{\omega_{\alpha}}w^2\dd t\bigg] \dd s
\\
\ge \int_0^a&(1-\widetilde C\alpha^{1-p}) |f'|^2+\dfrac{(p+q)^2-2(p+q)}{4s^2}|f|^2\dd s-\widetilde C\|u\|^2_{L^2(\Pi_\alpha)}
\end{aligned}
\end{equation}
for an appropriate $\widetilde C>0$ and $\alpha$ large enough.
Using the lower bounds in \eqref{minmax}, \eqref{born-op} and adjusting $\widetilde C>0$ we can choose $C>0$ such that
\begin{align*}
p_\alpha^-(u,u)  \ge   &(1- c\alpha^{1-p})\int_0^a\bigg((1-\widetilde C\alpha^{1-p})|f'|^2+\bigg[\dfrac{(p+q)^2-2(p+q)}{4s^2}
\\
&-\dfrac{\alpha^{p-1} }{(1- c\alpha^{1-p})s^{p}}-\dfrac{\alpha^{q-1}}{(1- c\alpha^{1-p})s^{q}}\bigg]|f|^2
\\
&+\bigg[\frac{\pi^2\alpha^{2p-2}}{4s^{2p}}-\dfrac{2\alpha^{p-1}}{(1- c\alpha^{1-p})s^{p}}-\dfrac{\alpha^{q-1}}{(1- c\alpha^{1-p})s^{q}}\bigg]\int_{\omega_{\alpha}}\!\!w^2\dd t\bigg)\dd s-\widetilde C\|u\|_{L^2(\Pi_\alpha)}^2
\\
\ge& (1- C\alpha^{1-p})\int_0^a \bigg(|f'|^2+\bigg[\dfrac{(p+q)^2-2(p+q)}{4s^2}
\\
&-\dfrac{\alpha^{p-1} }{(1- C\alpha^{1-p})s^{p}}-\dfrac{\alpha^{q-1}}{(1- C\alpha^{1-p})s^{q}}\bigg]|f|^2
\\
&+\!\left[\frac{\pi^2\alpha^{2p-2}}{4s^{2p}}-\dfrac{2\alpha^{p-1}}{(1- C\alpha^{1-p})s^{p}}-\dfrac{\alpha^{q-1}}{(1- C\alpha^{1-p})s^{q}}\right]\!\int_{\omega_{\alpha}}\!\!w^2\dd t\!\bigg)\!\dd s- C\|u\|_{L^2(\Pi_\alpha)}^2
\end{align*}
holds. One recognizes the quadratic form for $M_{\frac{\alpha^{p-1}}{1-C\alpha^{1-p}},\frac{\alpha^{q-1}}{1-C\alpha^{1-p}},a}$ as well as the quadratic form
\begin{equation*}
g_{\alpha,a}(w,w)=\int_0^a \left[\frac{\pi^2\alpha^{2p-2}}{4s^{2p}}-\dfrac{2\alpha^{p-1}}{(1- C\alpha^{1-p})s^{p}}-\dfrac{\alpha^{q-1}}{(1- C\alpha^{1-p})s^{q}}\right]\!\int_{\omega_{\alpha}}\!\!w^2\dd t\dd s
\end{equation*}
with $w\in D_0(p_\alpha^-)$. As a result of the min-max principle we have
\begin{align*}
&E_j(T_\alpha)\ge\inf_{\substack{S\subset D_0(p^-_\alpha) \\ {\rm dim\,  }S =j }} \, \sup_{\substack{u\in S\\ u\neq 0}}\,  \frac{p_\alpha^-(u,u)}{\quad \|u\|_{L^2(\Pi_\alpha)}^2}
\\
&\ge(1- C\alpha^{1-p})\!\inf_{\substack{S\subset C_c^\infty(0,a)\times D_0(p_\alpha^-) \\ {\rm dim\,  }S =j }} \, \sup_{\substack{(f,w)\in S\\ (f,w)\neq 0}}\dfrac{\langle (f,w), M_{\frac{\alpha^{p-1}}{1-C\alpha^{1-p}},\frac{\alpha^{q-1}}{1-C\alpha^{1-p}},a}\oplus G_{\alpha,a}(f,w)\rangle}{\| f\|_{L^2(0,a)}^2+\| w\|_{L^2(\Pi_\alpha)}^2}-C
\end{align*}
where $G_{\alpha,a}$ is the self-adjoint operator in $L^2(\Pi_\alpha)$ associated with $g_{\alpha,a}$. 
This implies
\begin{align*}
E_j(T_\alpha)&\ge (1-C\alpha^{1-p})\Lambda_j\left(M_{\frac{\alpha^{p-1}}{1-C\alpha^{1-p}},\frac{\alpha^{q-1}}{1-C\alpha^{1-p}},a}\oplus G_{\alpha,a}\right)- C.
\end{align*}
To end the proof we have to show that
\begin{equation}\label{mult op}
\Lambda_j\left(M_{\frac{\alpha^{p-1}}{1- C\alpha^{1-p}},\frac{\alpha^{q-1}}{1- C\alpha^{1-p}},a}\oplus G_{\alpha,a}\right)=E_j\left(M_{\frac{\alpha^{p-1}}{1- C\alpha^{1-p}},\frac{\alpha^{q-1}}{1- C\alpha^{1-p}},a}\right)
\end{equation}
for large $\alpha$. 
Making use of the unitary equivalence \eqref{uemodel} and Lemma \ref{modelop1} (by suitably adjusting constants) we get
\begin{equation}
\begin{aligned}
E_j\left(M_{\frac{\alpha^{p-1}}{1- C\alpha^{1-p}},\frac{\alpha^{q-1}}{1- C\alpha^{1-p}},a}\right)&\leq 
E_j\left(A_{\frac{\alpha^{p-1}}{1- C\alpha^{1-p}},\frac{\alpha^{q-1}}{1-C\alpha^{1-p}}}\right)+K
\\\label{truncated uppper bound}
&\le E_j\left(A_{0,\frac{\alpha^{q-1}}{1- C\alpha^{1-p}}}\right)+K
\\
&\le\alpha^{2\frac{q-1}{2-q}} E_j\left(A_{0,1}\right)+o(\alpha^{2\frac{q-1}{2-q}})
\end{aligned}
\end{equation}
for large enough $\alpha$.
It is left to show that the bottom of the spectrum of $G_{\alpha,a}$ is bigger than \eqref{truncated uppper bound}. 
Since $G_{\alpha,a}$ is a multiplication operator, it is sufficient to show that there exists $\alpha_0>0$ and $E(\alpha_0)>E_j(A_{0,1})$ such that for all $\alpha>\alpha_0$ there holds 
\[
\min_{s\in(0,a)}\frac{\pi^2\alpha^{2p-2}}{4s^{2p}}-\frac{2\alpha^{p-1}}{(1-C\alpha^{1-p})s^p}-\frac{\alpha^{q-1}}{(1-C\alpha^{1-p})s^q}>\alpha^{2\frac{q-1}{2-q}}\,E(\alpha_0).
\]
We estimate the minimum from below in the following way: There exists $c_0>0$ such that
\begin{align*}
\min_{s\in(0,a)}\frac{\pi^2\alpha^{2p-2}}{4s^{2p}}&-\frac{2\alpha^{p-1}}{(1-C\alpha^{1-p})s^p}-\frac{\alpha^{q-1}}{(1-C\alpha^{1-p})s^q}
\\
&\ge\min_{s\in(0,a)}\frac{c_0\alpha^{2p-2}}{s^{2p}}-\frac{\alpha^{q-1}}{(1-C\alpha_0^{1-p})s^q}
\\
&=\alpha^{2\frac{q-1}{2-q}}\min_{s\in(0,a)}\frac{c_0\alpha^{2p-2-2\frac{q-1}{2-q}}}{s^{2p}}-\frac{\alpha^{q-1-2\frac{q-1}{2-q}}}{(1-C\alpha_0^{1-p})s^q}
\\
&\ge\alpha^{2\frac{q-1}{2-q}}\min_{s\in(0,\infty)}\frac{c_0\alpha^{2p-2-2\frac{q-1}{2-q}}}{s^{2p}}-\frac{\alpha^{q-1-2\frac{q-1}{2-q}}}{(1-C\alpha_0^{1-p})s^q}
\\
&=\alpha^{2\frac{q-1}{2-q}}\min_{s\in(0,\infty)}\frac{c_0\alpha^{2p-2-2\frac{q-1}{2-q}-\frac{2p}{q}(q-1-2\frac{q-1}{2-q})}}{s^{2p}}-\frac{1}{(1-C\alpha_0^{1-p})s^q}
\\
&=\alpha^{2\frac{q-1}{2-q}}\min_{s\in(0,\infty)}\frac{c_0\alpha^{\frac{2p}{q}-2+(\frac{4p}{q}-2)\frac{q-1}{2-q}}}{s^{2p}}-\frac{1}{(1-C\alpha_0^{1-p})s^q}
\\
&\ge\alpha^{2\frac{q-1}{2-q}}\min_{s\in(0,\infty)}\frac{c_0\alpha_0^{\frac{2p}{q}-2+(\frac{4p}{q}-2)\frac{q-1}{2-q}}}{s^{2p}}-\frac{1}{(1-C\alpha_0^{1-p})s^q}.
\end{align*}
Remark that the last inequality holds due to the fact that the exponent $\frac{2p}{q}-2-(\frac{4p}{q}-2)\frac{q-1}{2-q}$ is positive, which also can be shown by an easy calculation.
Furthermore we have
\[
\min_{s\in(0,\infty)}\frac{c_0\alpha_0^{\frac{2p}{q}-2+(\frac{4p}{q}-2)\frac{q-1}{2-q}}}{s^{2p}}-\frac{1}{(1-C\alpha_0^{1-p})s^q}=:E(\alpha_0) \rightarrow 0\quad \text{as }\alpha_0\rightarrow\infty.
\]
That means we can choose $\alpha_0$ large enough to get $E_j(A_{0,1})<E(\alpha_0)$. Which in return yields the claim in \eqref{mult op}.
\end{proof}

\begin{prop}\label{prop8}
	For any $j\in\NN$ there exist $\alpha_0>$ and $\widetilde C>0$ such that
	\[
	E_j(T_\alpha)\ge \alpha^{2\frac{q-1}{2-q}}E_j(A_{0,1})-\widetilde C\left(\alpha^{2\frac{q-1}{2-q}-(p-1)}+\alpha^{2\frac{q-1}{2-q}-\frac{q-p}{2-q}}\right)
	\]
	for all $\alpha>\alpha_0$.
\end{prop}
\begin{proof}
Lemma \ref{lem9} and \eqref{model lowbd} show
\begin{equation}\label{intermediate}
\begin{aligned}
E_j(T_\alpha)&\ge (1-C\alpha^{1-p})E_j\left(M_{\frac{\alpha^{p-1}}{1-C\alpha^{1-p}},\frac{\alpha^{q-1}}{1-C\alpha^{1-p}},a}\right)
\\
&\ge (1-C\alpha^{1-p}) E_j\left(A_{\frac{\alpha^{p-1}}{1-C\alpha^{1-p}},\frac{\alpha^{q-1}}{1-C\alpha^{1-p}}}\right).
\end{aligned}
\end{equation}
Due to the unitary equivalence \eqref{uemodel}, we get
\begin{equation*}
E_j\left(A_{\frac{\alpha^{p-1}}{1-C\alpha^{1-p}},\frac{\alpha^{q-1}}{1-C\alpha^{1-p}}}\right)=\left(\frac{\alpha^{q-1}}{1-C\alpha^{1-p}}\right)^\frac{2}{2-q} E_j\left(A_{\frac{\alpha^{p-1}}{1-C\alpha^{1-p}}\left(\frac{\alpha^{q-1}}{1-C\alpha^{1-p}}\right)^\frac{p-2}{2-q},1}\right).
\end{equation*}
We want to compare the eigenvalue on the right hand side  with $E_j(A_{0,1})$. Due to the min-max principle it is sufficient enough to estimate the corresponding quadratic form from below. There exists $\widetilde c>0$ such that for any $f\in C_c^\infty (0,\infty)$ there holds
\begin{align*}
a&_{\frac{\alpha^{p-1}}{1-C\alpha^{1-p}}\left(\frac{\alpha^{q-1}}{1-C\alpha^{1-p}}\right)^\frac{p-2}{2-q},1}(f,f)
\\
=&\int_0^\infty |f'|^2 +\left[\frac{(p+q)^2-2(p+q)}{4s^2}-\frac{\alpha^{\frac{q-p}{q-2}}}{(1-C\alpha^{1-p})^{\frac{q-p}{q-2}}s^p}-\frac{1}{s^q}\right]|f|^2\dd s
\\
\ge& \int_0^1 |f'|^2 +\left[\frac{(p+q)^2-2(p+q)}{4s^2}-\frac{1+\widetilde c\alpha^{\frac{q-p}{q-2}}}{s^q}\right]|f|^2\dd s
\\
&+\int_1^\infty |f'|^2 +\left[\frac{(p+q)^2-2(p+q)}{4s^2}-\frac{1}{s^q}\right]|f|^2\dd s-\widetilde c\alpha^{\frac{q-p}{q-2}}\|f\|_{L^2(1,\infty)}^2
\\
\ge& \int_0^\infty |f'|^2 +\left[\frac{(p+q)^2-2(p+q)}{4s^2}-\frac{1+\widetilde c\alpha^{\frac{q-p}{q-2}}}{s^q}\right]|f|^2\dd s-\widetilde c\alpha^{\frac{q-p}{q-2}}\|f\|_{L^2(0,\infty)}^2.
\end{align*}
Using this result and applying the unitary equivalence \eqref{uemodel} one more time we arrive at
\begin{align*}
&E_j\left(A_{\frac{\alpha^{p-1}}{1-C\alpha^{1-p}}\left(\frac{\alpha^{q-1}}{1-C\alpha^{1-p}}\right)^\frac{p-2}{2-q},1}\right)\ge E_j\left(A_{0,1+\tilde c \alpha^{\frac{q-p}{q-2}}}\right)-\widetilde c\alpha^{\frac{q-p}{q-2}}
\\
&=(1+\widetilde c\alpha^{\frac{q-p}{q-2}})^{\frac{2}{2-q}}E_j(A_{0,1})-\widetilde c \alpha^{\frac{q-p}{q-2}}\ge E_j(A_{0,1})-\widetilde C\alpha^{\frac{q-p}{q-2}},
\end{align*}
for $\widetilde C>0$ large enough. Combining this with \eqref{intermediate} and adjusting $\widetilde C$ gives the desired estimate
\begin{equation*}
E_j(T_\alpha)\ge \alpha^{2\frac{q-1}{2-q}}E_j(A_{0,1})-\widetilde C\left(\alpha^{2\frac{q-1}{2-q}-(p-1)}+\alpha^{2\frac{q-1}{2-q}-\frac{q-p}{2-q}}\right).
\end{equation*}
\end{proof}

\section{End of proof of Theorem \ref{thm1}}\label{ssthm1}

So far we have shown that $E_j(T_\alpha)\approx E_j(A_{0,1})\alpha^{2\frac{q-1}{2-q}}$, which corresponds to the leading term of $Q_{\alpha\Omega}^1$ according to Theorem \ref{thm1}, if one uses the unitary equivalence $Q_\Omega^\alpha\simeq \alpha^2Q_{\alpha\Omega}^1$. Therefore, in order to conclude the proof of Theorem \ref{thm1} it remains to show that the eigenvalues of $Q_{\alpha\Omega}^1$ and $T_\alpha$
with the same numbers are close to each other. We will use \eqref{sandwichRobin} as a starting point, to estimate $E_j(Q_{\alpha\Omega}^1)$ by means of $E_j(T_\alpha)$.

\begin{lemma}\label{lem10}
For any $j\in\NN$ and $\alpha>0$ the inequality $E_j(R_1^{D,\alpha\Omega_\delta})\leq E_{j}(T_\alpha)$ holds.
\end{lemma}
\begin{proof}
Let $J:L^2(V_\alpha)\to L^2(\alpha\Omega_\delta)$ be the operator of extension by zero, then $J$ is a linear isometry
with $JD(t_\alpha)\subset D(r_1^{D,\alpha\Omega_\delta})$ and with $r_1^{D,\alpha\Omega_\delta}(Ju,Ju)=t_\alpha(u,u)$ for all $u\in D(t_\alpha)$, and the result follows directly by the min-max principle.
\end{proof}
Now we are able to prove that the eigenvalues of $R_{\alpha}^{N,\Omega_\delta}\oplus K_{\alpha}^{N,\Theta_\delta}$ are determined by $R_{\alpha}^{N,\Omega_\delta}$ for large $\alpha$.
\begin{lemma}\label{10.5}
Let $j\in\N$, then there exists $\alpha_0>0$ such that for all $\alpha>\alpha_0$
\[
E_j(R_{\alpha}^{N,\Omega_\delta}\oplus K_{\alpha}^{N,\Theta_\delta})=E_j(R_{\alpha}^{N,\Omega_\delta}).
\]
\end{lemma}
\begin{proof}
For the Robin-Laplacian $Q^{\alpha}_{\Theta_{\delta}}$ there holds $E_1(Q^{\alpha}_{\Theta_{\delta}})\ge -K\alpha^2$, since $\Theta_{\delta}$ is a bounded Lipschitz domain. Therefore, by the min-max principle, $E_1(K_{\alpha}^{N,\Theta_\delta})\ge -K\alpha^2$. The combination of Lemma \ref{lem10} and Proposition \ref{prop7} shows $E_j(R_{\alpha}^{D,\Omega_\delta})<-K\alpha^2$ for large $\alpha$. Applying the min-max principle once more shows $E_j(R_{\alpha}^{N,\Omega_\delta}\oplus K_{\alpha}^{N,\Theta_\delta})=E_j(R_{\alpha}^{N,\Omega_\delta})$, for large $\alpha$.
\end{proof}
Lemma \ref{10.5} improves inequality \eqref{sandwichRobin} to
\begin{equation}\label{sandwich}
E_j(R_{\alpha}^{N,\Omega_{\delta}})\le E_j(Q^{\alpha}_{\Omega})\le E_j(R_{\alpha}^{D,\Omega_{\delta}}).
\end{equation}
\begin{lemma}\label{lem11}
Let $j\in\NN$ then there exist $k>0$ and $\alpha_{0}>0$ such that for all $\alpha>\alpha_{0}$ there holds
$E_{j}(R_1^{N,\alpha\Omega_\delta})\ge E_j (T_\alpha)-k$.
\end{lemma}
\begin{proof}

Analogously to Lemma \ref{modelop1} we use a IMS partition of unity. Let $\Phi_1,\Phi_2\in C^\infty(\R)$ such that $\Phi_1^2+\Phi_2^2=1,\ \Phi_1(s)=0\text{ for }s>a,\Phi_2(s)=0 \text{ for }s<a/2$ and define $k:=\|\Phi_1'\|_{\infty}^2+\|\Phi_2'\|_{\infty}^2,\ \chi_{j}(x_1,x_2,x_3)=\Phi_j(x_3)$ with $j\in\{1,2\}$. It follows that
\begin{align*}
r_1^{N,\alpha\Omega_\delta}(u,u)&=r_1^{N,\alpha\Omega_\delta}(\chi_1 u,\chi_1 u)+r_1^{N,\alpha\Omega_\delta}(\chi_2 u,\chi_2 u)-\int_{\alpha\Omega_\delta}(|\nabla\chi_1|^2+|\nabla \chi_2|^2)u^2\dd x
\\
&\ge r_1^{N,\alpha\Omega_\delta}(\chi_1 u,\chi_1 u)+r_1^{N,\alpha\Omega_\delta}(\chi_2 u,\chi_2 u)-k\|u\|^2_{L^2(\alpha\Omega_{\delta})}.
\end{align*}
We define
\begin{align*}
w_\alpha(u,u):&=\int_{\widetilde W_\alpha}|\nabla u|^2\dd x-\int_{\partial_0 \widetilde W_\alpha}u^2\dd \sigma,
\\
\widetilde W_\alpha:&=(\alpha\Omega_\delta)\cap\left\{\frac{a}{2}<x_3\right\},
\\
\partial_0 \widetilde W_\alpha:&=\partial(\alpha\Omega_\delta)\cap\left\{\frac{a}{2}<x_3<\alpha\delta\right\},
\\
D(w_\alpha)&=\widehat{H}_0^1(\widetilde W_\alpha):=\left\{u\in H^1(\widetilde W_{\alpha}):u\left(\cdot,\frac{a}{2}\right)=0\right\}.
\end{align*} 
Then we have
\begin{align*}
\|\chi_1 u\|_{L^2(\alpha\Omega_\delta)}&=\|\chi_1 u\|_{L^2(V_\alpha)},\quad \|\chi_2 u\|_{L^2(\alpha\Omega_\delta)}=\|\chi_2 u\|_{L^2(\widetilde W_\alpha)},
\\
r_1^{N,\alpha\Omega_\delta}(\chi_1 u,\chi_1 u)&=t_\alpha(\chi_1 u,\chi_1 u),\quad r_1^{N,\alpha\Omega_\delta}(\chi_2 u,\chi_2 u)=w_\alpha(\chi_2 u,\chi_2 u).
\end{align*}
Plugging these results into the min-max principle yields
\begin{equation}\label{ims2}
\begin{aligned}
E_j(R_1^{N,\alpha\Omega_\delta})+k&=\inf_{\substack{S\subset D\big(r_1^{N,\alpha\Omega_\delta}\big)\\ \dim S=j}} \sup_{\substack{u\in S\\ u\ne 0}} \dfrac{r_1^{N,\alpha\Omega_\delta}(u,u) +k\|u\|^2_{L^2(\alpha\Omega_\delta)}}{\|u\|^2_{L^2(\alpha\Omega_\delta)}}
\\
&\ge\inf_{\substack{S\subset D\left(r_1^{N,\alpha\Omega_\delta}\right)\\ \dim S=j}} \sup_{\substack{u\in S\\ u\ne 0}} \dfrac{r_1^{N,\alpha\Omega_\delta}(\chi_1 u,\chi_1 u) +r_1^{N,\alpha\Omega_\delta}(\chi_2 u,\chi_2 u)}{\|\chi_1 u\|^2_{L^2(\alpha\Omega_\delta)}+\|\chi_2 u\|^2_{L^2(\alpha\Omega_\delta)}}
\\
&=\inf_{\substack{S\subset D\left(r_1^{N,\alpha\Omega_\delta}\right)\\ \dim S=j}} \sup_{\substack{u\in S\\ u\ne 0}} \dfrac{t_\alpha(\chi_1 u,\chi_1 u) +w_\alpha(\chi_2 u,\chi_2 u)}{\|\chi_1 u\|^2_{L^2(V_\alpha)}+\|\chi_2 u\|^2_{L^2(\widetilde W_\alpha)}}
\\
&\ge \inf_{\substack{S\subset D(t_\alpha)\times D(w_\alpha)\\ \dim S=j}} \sup_{\substack{(u_1,u_2)\in S\\ (u_1,u_2)\ne 0}} \dfrac{t_\alpha(u_1,u_1) +w_\alpha(u_2,u_2)}{\|u_1\|^2_{L^2(V_\alpha)}+\|u_2\|^2_{L^2(\widetilde W_\alpha)}}
\\
&=E_j(T_\alpha\oplus W_\alpha).
\end{aligned}
\end{equation}
Our aim now is to find a lower bound for $W_\alpha$, which shows that $E_j(T_\alpha\oplus W_\alpha)=E_j(T_\alpha)$ for large enough $\alpha$. There holds
\begin{align*}
&\int_{\widetilde W_\alpha}|\nabla u|^2\dd x-\int_{\partial_0 \widetilde W_\alpha}u^2\dd \sigma =\int_{\frac{a}{2}}^{\alpha\delta}\bigg[\int_{-\alpha^{1-p}x_3^p}^{\alpha^{1-p}x_3^p}\!\int_{-\alpha^{1-q}x_3^q}^{\alpha^{1-q}x_3^q}|\nabla u|^2\dd x_2\dd x_1
\\
&-\sqrt{1+q^2\alpha^{2-2q}x_3^{2q-2}}\!\!\int_{-\alpha^{1-p}x_3^p}^{\alpha^{1-p}x_3^p}|u(x_1,\alpha^{1-q}x_3^q,x_3)|^2-|u(x_1,-\alpha^{1-q}x_3^q,x_3)|^2\dd x_1
\\
&-\!\!\sqrt{1+p^2\alpha^{2-2p}x_3^{2p-2}}\!\!\int_{-\alpha^{1-q}x_3^q}^{\alpha^{1-q}x_3^q}|u(\alpha^{1-p}x_3^p,x_2,x_3)|^2\!\!-|u(-\alpha^{1-p}x_3^p,x_2,x_3)|^2\dd x_2\bigg]\dd x_3
\\
=&\int_{\frac{a}{2}}^{\alpha\delta}\bigg[\int_{-\alpha^{1-p}x_3^p}^{\alpha^{1-p}x_3^p}\bigg\{\int_{-\alpha^{1-q}x_3^q}^{\alpha^{1-q}x_3^q}|\partial_{x_2} u|^2\dd x_2
\\
&-\sqrt{1+q^2\alpha^{2-2q}x_3^{2q-2}}\left(|u(x_1,\alpha^{1-q}x_3^q,x_3)|^2+|u(x_1,-\alpha^{1-q}x_3^q,x_3)|^2\right)\bigg\}\dd x_1
\\
&+\int_{-\alpha^{1-q}x_3^q}^{\alpha^{1-q}x_3^q}\bigg\{\int_{-\alpha^{1-p}x_3^p}^{\alpha^{1-p}x_3^p}|\partial_{x_1} u|^2\dd x_1
\\
&-\sqrt{1+p^2\alpha^{2-2p}x_3^{2p-2}}\left(|u(\alpha^{1-p}x_3^p,x_2,x_3)|^2+|u(-\alpha^{1-p}x_3^p,x_2,x_3)|^2\right)\bigg\}\dd x_2
\\
&+\int_{-\alpha^{1-q}x_3^q}^{\alpha^{1-q}x_3^q}\int_{-\alpha^{1-p}x_3^p}^{\alpha^{1-p}x_3^p}|\partial_{x_3} u|^2\dd x_1 \dd x_2\bigg] \dd x_3
\\
\ge&\int_{\frac{a}{2}}^{\alpha\delta}\bigg[\int_{-\alpha^{1-p}x_3^p}^{\alpha^{1-p}x_3^p}\bigg\{\int_{-\alpha^{1-q}x_3^q}^{\alpha^{1-q}x_3^q}|\partial_{x_2} u|^2\dd x_2
\\
&-\sqrt{1+q^2\alpha^{2-2q}x_3^{2q-2}}\left(|u(x_1,\alpha^{1-q}x_3^q,x_3)|^2+|u(x_1,-\alpha^{1-q}x_3^q,x_3)|^2\right)\bigg\}\dd x_1
\\
&+\int_{-\alpha^{1-q}x_3^q}^{\alpha^{1-q}x_3^q}\bigg\{\int_{-\alpha^{1-p}x_3^p}^{\alpha^{1-p}x_3^p}|\partial_{x_1} u|^2\dd x_1
\\
&-\sqrt{1+p^2\alpha^{2-2p}x_3^{2p-2}}\left(|u(\alpha^{1-p}x_3^p,x_2,x_3)|^2+|u(-\alpha^{1-p}x_3^p,x_2,x_3)|^2\right)\bigg\}\dd x_2\bigg] \dd x_3
\\
\ge& \int_{\frac{a}{2}}^{\alpha\delta}\bigg[E_1\left(B_{\alpha^{1-p}x_3^p,\sqrt{1+p^2\alpha^{2-2p}x_3^{2p-2}}}\otimes 1+1\otimes B_{\alpha^{1-q}x_3^q,\sqrt{1+q^2\alpha^{2-2q}x_3^{2q-2}}}\right)
\\
&\cdot\int_{-\alpha^{1-q}x_3^q}^{\alpha^{1-q}x_3^q}\int_{-\alpha^{1-p}x_3^p}^{\alpha^{1-p}x_3^p} u^2 \dd x_1\dd x_2\bigg]\dd x_3
\\
\ge& \Lambda \|u\|^2_{L^2(\widetilde W_\alpha)}
\\
&\text{with }\Lambda:=\!\!\inf_{x_3\in(\frac{a}{2},\alpha\delta)}E_1\left(B_{\alpha^{1-p}x_3^p,\sqrt{1+p^2\alpha^{2-2p}x_3^{2p-2}}}\otimes 1+1\otimes B_{\alpha^{1-q}x_3^q,\sqrt{1+q^2\alpha^{2-2q}x_3^{2q-2}}}\right).
\end{align*}
By Lemma \ref{lem-1}, it follows
\begin{align*}
E_1&\left(B_{\alpha^{1-p}x_3^p,\sqrt{1+p^2\alpha^{2-2p}x_3^{2p-2}}}\otimes 1+1\otimes B_{\alpha^{1-q}x_3^q,\sqrt{1+q^2\alpha^{2-2q}x_3^{2q-2}}}\right)
\\
\ge&-\frac{\alpha^{p-1}\sqrt{1+p^2\alpha^{2-2p}x_3^{2p-2}}}{x_3^p}-C\left(1+p^2\alpha^{2-2p}x_3^{2p-2}\right)
\\
&-\frac{\alpha^{q-1}\sqrt{1+q^2\alpha^{2-2q}x_3^{2q-2}}}{x_3^q}-C\left(1+q^2\alpha^{2-2q}x_3^{2q-2}\right)
\\
\ge& -\frac{\alpha^{p-1}2^p \sqrt{1+p^2\delta^{2p-2}}}{a^p}-C\left(1+p^2\delta^{2p-2}\right)
\\
&-\frac{\alpha^{q-1}2^q\sqrt{1+q^2\delta^{2q-2}}}{a^q}-C\left(1+q^2\delta^{2q-2}\right)
\\
\ge& -C_0\alpha^{q-1}
\end{align*}
for appropriate $C_0,C>0$ and sufficiently large $\alpha$. By Proposition \ref{prop7} and Proposition \ref{prop8} we know $E_j(T_\alpha)\approx -\alpha^{2\frac{q-1}{2-q}}$, which shows $E_j(T_\alpha\oplus W_\alpha)=E_j(T_\alpha)$ for large enough $\alpha$. Combining this and \eqref{ims2} concludes the proof.
\end{proof} 
\begin{prop}\label{prop9}
Let $j\in\N$ be fixed, then the $j$th eigenvalue of $Q_{\alpha\Omega}^1$ satisfies
\begin{equation*}
E_j(Q_{\alpha\Omega}^1)=E_j(T_\alpha)+O(1)\quad \text{as }\alpha\rightarrow\infty.
\end{equation*}
\end{prop}
\begin{proof}
Let $j\in \N$ be fixed. By combining \eqref{sandwich} and Lemma \ref{lem10} and Lemma \ref{lem11} we obtain $E_j(Q_{\alpha\Omega}^1)=E_j(T_\alpha)+O(1)$ as $\alpha\rightarrow\infty$.
\end{proof}
The results from above allow us to prove Theorem \ref{thm1}.
\begin{proof}[Proof of Theorem \ref{thm1}]
Proposition \ref{prop9} together with the unitary equivalence $Q^\alpha_\Omega\simeq\alpha^2 Q^1_{\alpha\Omega}$ and the asymptotics of $E_j(T_\alpha)$ given by Proposition \ref{prop7} and Proposition \ref{prop8} prove Theorem \ref{thm1}.
\end{proof}
\appendix
\section*{Appendix}

\section{Density in Sobolev spaces on domains with non-isotropic peaks}\label{appa}

To avoid technical difficulties we don't work with $\widehat{H}^1_0(\Omega_\delta)$. Since we use the min-max principle it is sufficient to work with dense subspaces of $\widehat{H}^1_0(\Omega_\delta)$. For that purpose let $\Omega_\delta$ as in \eqref{omegadelta} and define
\begin{align*}
	\label{cinfi}
	C^\infty_{(0,\delta)}(\Bar \Omega_\delta):=\big\{ u\in C^\infty(\Bar \Omega_\delta): \exists\, [b,c]&\subset (0, \delta)\;\text{such that}
\\	
	&\text{$u(x_1,x_2,x_3)=0$ for $x_3\notin [b,c]$} \big\}.
\end{align*}
Then we have the following density result.
\begin{lemma}\label{prop4dens}
	The subspace $ C^\infty_{(0,\delta)}(\Bar \Omega_\delta)$ is dense in $\widehat{H}^1_0(\Omega_\delta)$ with respect to the $H^1$ norm.
\end{lemma}
\begin{proof}
Since $\Omega_\delta$ has a $C^0$ boundary, the space $C^\infty(\Bar \Omega_\delta)$ is dense in $H^1(\Omega_\delta)$ according to \cite[Theorem 1 in Sec. 1.4.2]{bad domain}. Therefore $C^\infty(\Bar \Omega_\delta)\cap\widehat{H}^1_0(\Omega_\delta)$ is dense in $\widehat{H}^1_0(\Omega_\delta)$ and it is left to show that every function in $C^\infty(\Bar \Omega_\delta)\cap\widehat{H}^1_0(\Omega_\delta)$ can be approximated by functions in $C^\infty_{(0,\delta)}(\Bar \Omega_\delta)$.

Let $u\in C^\infty(\Bar \Omega_\delta)\cap\widehat{H}^1_0(\Omega_\delta)$ and $\chi:\R\rightarrow\R$ be a smooth function with $\chi(s)=0$ for $s<\frac{1}{2}$ and $\chi(s)=1$ for $s>1$. For $\eps>0$ let $u_{\eps}$ be the function on $\Omega_\delta$ given by
\begin{equation*}
u_\eps(x)=u(x)\chi\left(\frac{x_3}{\eps}\right)\chi\left(\frac{\delta-x_3}{\eps}\right).
\end{equation*}
Then there holds
\begin{align*}
\|u-u_\eps\|_{H^1(\Omega_\delta)}^2\le &C \int_{\Omega_\delta}\left(|u|^2+|\nabla u|^2\right)\left(1-\chi\left(\frac{x_3}{\eps}\right)\chi\left(\frac{\delta-x_3}{\eps}\right)\right)^2\dd x
\\
&+\frac{C}{\eps^2}\int_{\Omega_\delta\cap\{x_3<\eps\}}|u|^2\dd x+\frac{C}{\eps^2}\int_{\Omega_\delta\cap\{x_3>\delta-\eps\}}|u|^2\dd x=I_1+I_2+I_3.
\end{align*}
Since $u\in H^1(\Omega_\delta)$ the first integral $I_1$ converges to $0$ as $\eps\rightarrow 0$, by the dominated convergence theorem. For $I_2$ we estimate
\begin{align*}
\int_{\Omega_\delta\cap\{x_3<\eps\}}|u|^2\dd x&\le \|u\|^2_{\infty}\int_{\Omega_\delta\cap\{x_3<\eps\}}\dd x=\|u\|^2_\infty\int_0^\eps \int_{(-x_3^p,x_3^p)\times(-x_3^q,x_3^q)}\dd x
\\
&=4\|u\|^2_\infty\int_0^\eps x_3^{p+q}\dd x_3 =\frac{4\|u\|^2_\infty}{p+q+1}\eps^{p+q+1}
\end{align*}
and this implies $I_2\rightarrow 0$ as $\eps\rightarrow 0$. Lastly for $I_3$ we remark that
\begin{equation*}
u(x_1,x_2,x_3)=-\int^\delta_{x_3}\partial_{x_3}u(x_1,x_2,t)\dd t
\end{equation*}
and then using H\"older's inequality we arrive at
\begin{align*}
|u(x_1,x_2,x_3)|^2=\left|\int_{x_3}^\delta\partial_{x_3}u(x_1,x_2,t)\dd t\right|^2&\le (\delta-x_3)\int_{x_3}^\delta \left|\partial_{x_3}u(x_1,x_2,t)\right|^2\dd t
\\
&\le (\delta-x_3)\int_{x_3}^\delta|\nabla u(x_1,x_2,t)|^2\dd t.
\end{align*}
Now we use this inequality for $I_3$ in the following way
\begin{align*}
I_3&\le \frac{C}{\eps^2}\int_{\Omega_\delta\cap\{x_3>\delta-\eps\}}(\delta-x_3)\int_{x_3}^\delta |\nabla u(x_1,x_2,t)|^2\dd t\dd x
\\
&\le\frac{C}{\eps^2} \int_{\delta-\eps}^\delta(\delta-x_3)\int_{x_3}^\delta\int_{(-x_3^p,x_3^p)\times(-x_3^q,x_3^q)} |\nabla u(x_1,x_2,t)|^2\dd x_1\dd x_2 \dd t\dd x_3
\\
&\le\frac{C}{\eps^2}\int_{\delta-\eps}^\delta(\delta-x_3)\dd x_3\int_{\delta-\eps}^\delta\int_{(-t^p,t^p)\times(-t^q,t^q)} |\nabla u(x_1,x_2,t)|^2\dd x_1\dd x_2 \dd t
\\
&=\frac{C}{2}\int_{\Omega_\delta\cap\{x_3>\delta-\eps\}}|\nabla u|^2\dd x,
\end{align*}
which shows $I_3\rightarrow 0$ as $\eps\rightarrow0$ and that concludes the proof.
\end{proof}

\section{Closedness of $q_\Omega^\alpha$}\label{appb}
For a Lipschitz domain it is well known that the quadratic form associated with the Robin-Laplacian is closed. To show this, one can use the trace inequality from \cite[Theorem 1.5.1.10]{grisvard}. However we can not use this result immediately, since $\Omega$ has a peak. Fortunately this inequality extends to domains with non-isotropic peaks. But let us first show the semiboundedness of $q_\Omega^\alpha$.
\begin{lemma}\label{semibdd}
The quadratic form $q_\Omega^\alpha$ is semibounded from below, for any $\alpha>0$ and any open set $\Omega$ which satisfies \eqref{non iso peak} and \eqref{lipschitz part}.
\end{lemma}
\begin{proof}
Whether or not a quadratic form $t$ is semibounded from below, one can still calculate the min-max value 
\[
\Lambda_1(t):=\inf_{\substack{u\in D(t)\\ u\neq0}}\frac{t(u,u)}{\| u\|_{\cH}^2}.
\]
Remark that a quadratic form is semibounded from below if $\Lambda_1(t)>-\infty$. Therefore we need to show that $\Lambda_1(q_\Omega^\alpha)>-\infty$. We can argue as in Subsection \ref{lowbdd} to get that there exists $\alpha_0>0$ such that for all $\alpha\ge\alpha_0$ there holds
\[
\Lambda_1(t_\alpha)\ge \alpha^{2\frac{q-1}{2-q}}E_1(A_{0,1})+O\left(\alpha^{2\frac{q-1}{2-q}-(p-1)}+\alpha^{2\frac{q-1}{2-q}-\frac{q-p}{2-q}}\right).
	\]
Furthermore, we can use the same arguments as in Subsection \ref{isolateing peak} and Subsection \ref{ssthm1} to show that for sufficiently large $\alpha$ one has  $\Lambda_1(q_\Omega^\alpha)\ge\Lambda_1(r_\alpha^{N,\Omega_\delta})\ge\alpha^2\Lambda_1(t_\alpha)$.
Hence, combining these inequalities, we can find $K>0$ such that $\Lambda_1(q_\Omega^\alpha)\ge -K\alpha^{\frac{2}{2-q}}$ for all $\alpha\ge\alpha_0$, which proves the semiboundedness for all $\alpha\ge\alpha_0$. Now let $\widetilde\alpha<\alpha_0$, then one easily sees
\[
q_\Omega^{\tilde\alpha}(u,u)\geq q_\Omega^{\alpha_0}(u,u) \quad \text{for all $u\in H^1(\Omega)$}.
\]
But this immediately implies $\Lambda_1(q_\Omega^{\tilde{\alpha}})\ge\Lambda_1(q_\Omega^{\alpha_0})\ge -K\alpha_0^{\frac{2}{2-q}}$. All in all we showed that $q_\Omega^\alpha$ is semibounded from below for all $\alpha>0$.
\end{proof}
To conclude the closedness of $q_\Omega^\alpha$ we want to extent the trace inequality from \cite[Theorem 1.5.1.10]{grisvard} to domains with non-isotropic peaks.
\begin{prop}
For any $\delta>0$ there exists $C_\delta>0$ such that
\[
\|u\|_{L^2(\partial\Omega)}^2\le\delta\|\nabla u\|_{L^2(\Omega)}^2+C_\delta\|u\|_{L^2(\Omega)}^2\quad \text{for any $u\in H^1(\Omega)$.}
\]
In particular the norm induced by $q_\Omega^\alpha$ is equivalent to the standard $H^1$-norm and therefore $q_\Omega^\alpha$ is closed.
\end{prop}
\begin{proof}
Due to Lemma \ref{semibdd}, for every $\alpha>0$ we can find a constant $K_\alpha>0$ such that
\[
\|\nabla u\|_{L^2(\Omega)}^2-\alpha\|u\|_{L^2(\partial\Omega)}^2\ge-K_\alpha\|u\|_{L^2(\Omega)}^2\quad \text{for all $u\in H^1(\Omega)$}.
\]
Rearranging this inequality yields
\[
\frac{1}{\alpha}\|\nabla u\|_{L^2(\Omega)}^2+\frac{K_\alpha}{\alpha}\|u\|_{L^2(\Omega)}^2\ge\|u\|_{L^2(\partial\Omega)}^2\quad \text{for all $u\in H^1(\Omega)$}.
\]
Now let $\delta>0$ be arbitrary and choose $\alpha>0$ such that $\delta=\alpha^{-1}$, then the claim follows, with $C_\delta=\frac{K_\alpha}{\alpha}$.
\end{proof}

\noindent {\bf Acknowledgements.} Many thanks to Konstantin Pankrashkin for suggesting the study of Robin-Laplacians on non-isotropic peaks and for helpful discussions on various issues related to this topic.


\begin{thebibliography}{999}

\small

\setlength\itemsep{-1pt}

\bibitem{bfk}
D. Bucur, P. Freitas, J. B. Kennedy: \emph{The Robin problem.} A. Henrot (Ed.): \emph{Shape optimization
and spectral theory.} De Gruyter Open, 2017, pp.~78–119.

\bibitem{bsim} H. L. Cycon, R. G. Froese, W. Kirsch, B. Simon: \emph{Schr\"odinger operators: with applications to quantum mechanics and global geometry.} Texts and Monographs in Physics. Springer Study Edition. Springer-Verlag, Berlin, 1987.

\bibitem{kennedy}
D. Daners, J. B. Kennedy: \emph{On the asymptotic behaviour of the eigenvalues of a Robin problem.} Differential and integral equations, 2010, Vol.23 (7/8) 659-669.

\bibitem{edmunds}
D. E. Edmunds, W. D. Evans: \emph{Spectral theory and differential operators.} Oxford: Clarendon Press, 1987.

\bibitem{exner}
P. Exner, A. Minakov,L. Parnovski: \emph{Asymptotic eigenvalue estimates for a Robin problem with a large parameter.} Portugaliae Mathematica 71.2 (2014): 141-156.

\bibitem{grisvard}
P. Grisvard: \emph{Elliptic problems in nonsmooth domains.} Pitman Publishing, 1985.

\bibitem{hp}
B. Helffer, K. Pankrashkin: \emph{Tunneling between corners for Robin Laplacians.}
J. London Math. Soc. {\bf 91} (2015) 225--248.

\bibitem{kato}
T. Kato: \emph{Perturbation theory for linear operators.}
Reprint of the 1980 edition. Classics in Mathematics. Springer-Verlag, Berlin, 1995.

\bibitem{konstantin}
M. Khalile, T. Ourmières-Bonafos, K. Pankrashkin: \emph{Effective operators for Robin eigenvalues in domains with corners.} Annales de l'Institut Fourier. Vol. 70. No. 5. 2020, pp. 2215-2301.

\bibitem{kp18}
M. Khalile, K. Pankrashkin: \emph{Eigenvalues of Robin Laplacians in infinite sectors.} Math. Nachr. {\bf 291} (2018) 928--965.

\bibitem{kov}  H. Kova\v{r}\'{\i}k, K. Pankrashkin: \emph{Robin eigenvalues on domains with peaks.}
J. Differential Equations {\bf 267} (2019) 1600--1630.

\bibitem{lacey}
A. A. Lacey, J. R. Ockendon, J. Sabina: \emph{Multidimensional reaction diffusion
equations with nonlinear boundary conditions}, SIAM J. Appl. Math. {\bf 58}:5 (1998) 1622--1647.

\bibitem{levitin}
M. Levitin, L. Parnovski: \emph{On the principal eigenvalue of a Robin problem with a large parameter.} Math. Nachr. {\bf 281}:2 (2008) 272--281.

\bibitem{bad domain}
V. G. Maz'ya, S. V. Poborchi: \emph{Differential functions on bad domains.}
World Scientific, 1997.

\bibitem{nazarov}
S. A. Nazarov, J. Taskinen.: \emph{Spectral anomalies of the Robin Laplacian in non-Lipschitz domains.} J. Math. Sci. Univ. Tokyo 20.1 (2013), pp. 27-90.

\bibitem{vogel}
K. Pankrashkin, M. Vogel: \emph{Asymptotics of Robin eigenvalues on sharp infinite cones.} J. Spectr. Theory 13 (2023), no. 1, pp. 201–241.

\bibitem{RS4}
M. Reed, B. Simon: \emph{Methods of modern mathematical physics. IV. Analysis of operators.} Academic Press, New York-London, 1978.

\end{thebibliography}
\end{document}